\def\siam{0} 
\ifnum\siam=1 
    \documentclass[review,onefignum,onetabnum]{siamart171218}
    
    \newsiamremark{remark}{Remark}
    \newsiamremark{hypothesis}{Hypothesis}
    \crefname{hypothesis}{Hypothesis}{Hypotheses}
    \newsiamthm{claim}{Claim}

    \headers{Linear poroelasticity with solid incompressibility}{Nicolas A. Barnafi, Andrés E. Rubiano, Ricardo Ruiz-Baier}
    \makeatletter
    \newcommand*{\addFileDependency}[1]{
      \typeout{(#1)}
      \@addtofilelist{#1}
      \IfFileExists{#1}{}{\typeout{No file #1.}}
    }
    \makeatother

    \usepackage{mathpazo}
    \usepackage{amssymb,amsfonts}

\else
    \documentclass{article}
    \usepackage{amsthm}
    \usepackage{fullpage} 
    \usepackage{mathpazo}
    \usepackage{amssymb,amsfonts,amsmath}
    \newtheorem{remark}{Remark}[section]
    \newtheorem{theorem}{Theorem}[section]
    \newtheorem{lemma}[theorem]{Lemma}

    \usepackage[colorlinks=true,breaklinks=true,linkcolor=lightblue,citecolor=lightblue]{hyperref} 
    \allowdisplaybreaks

    \numberwithin{equation}{section}
    \numberwithin{figure}{section}
    \numberwithin{table}{section}
\fi

\usepackage[T1]{fontenc}
\usepackage{graphicx,subcaption} 
\usepackage{isomath}
\usepackage{booktabs,todonotes,pgfplots,pgfplotstable}
\usepackage{mathtools} 
\usepackage[nolist]{acronym}

\renewcommand*{\vec}{\vectorsym}
\newcommand{\ten}{\tensorsym}

\DeclareMathOperator{\dive}{\mathrm{div}}
\DeclareMathOperator{\grad}{\nabla}
\DeclareMathOperator{\cof}{\mathrm{cof}}

\newcommand{\parder}[2]{\frac{\partial #1}{\partial #2}}

\newcommand\rH{\mathrm{H}}
\newcommand\rL{\mathrm{L}}
\newcommand\rV{\mathrm{V}}
\newcommand\rQ{\mathrm{Q}}
\newcommand\rP{\mathrm{P}}

\newcommand\bH{\mathbf{H}}
\newcommand\bL{\mathbf{L}}
\newcommand\bV{\mathbf{V}}
\newcommand\bP{\mathbf{P}}

\newcommand\cA{\mathcal{A}}
\newcommand\cB{\mathcal{B}}
\newcommand\cC{\mathcal{C}}
\newcommand\cD{\mathcal{D}}
\newcommand\cF{\mathcal{F}}
\newcommand\cG{\mathcal{G}}
\newcommand\cH{\mathcal{H}}
\newcommand\cI{\mathcal{I}}

\usepackage{listings}
\usepackage{xcolor}
\definecolor{codegreen}{rgb}{0,0.6,0}
\definecolor{codegray}{rgb}{0.5,0.5,0.5}
\definecolor{codepurple}{rgb}{0.58,0,0.82}
\definecolor{backcolour}{rgb}{0.95,0.95,0.92}
\definecolor{lightgreen}{rgb}{0.22,0.50,0.25}
\definecolor{lightblue}{rgb}{0.22,0.45,0.70}
\lstdefinestyle{mystyle}{
	backgroundcolor=\color{backcolour}, commentstyle=\color{codegreen},
	keywordstyle=\color{magenta},
	numberstyle=\tiny\color{codegray},
	stringstyle=\color{codepurple},
	basicstyle=\ttfamily\footnotesize,
	breakatwhitespace=false,         
	captionpos=b,                    
	keepspaces=true,                 
	numbers=none,                    
	numbersep=5pt,                  
	showspaces=false,                
	showstringspaces=false,
	showtabs=false,                  
	tabsize=2,
	framerule=1.5pt,
	rulecolor=\color{red!60!black},
	float
}
\title{Linear poroelasticity with solid incompressibility: consistent formulation and scalable numerical solution\thanks{This work has been supported by Center for Mathematical Modeling (CMM), Proyecto Basal FB210005, by FONDECYT Postdoctoral Grant 3230326,  by the Australian Research Council through the Future Fellowship grant FT220100496 and Discovery Project grant DP22010316, and by the Center of Advanced Study (CAS) at the Norwegian Academy of Science and Letters under the program \textit{Mathematical Challenges in Brain Mechanics}.}}
\author{Nicolás A. Barnafi$^{\dag,}$\thanks{Corresponding author.}\thanks{Instituto de Ingeniería Matemática y Computacional \& Facultad de Ciencias Biológicas, Pontificia Universidad Católica de Chile, Santiago, Chile; Centro de Modelamiento Matemático, Santiago, Chile. Email: nicolas.barnafi@uc.cl} \and Andrés E. Rubiano\thanks{School of Mathematics, Monash University, 9 Rainforest Walk, Clayton 3800 VIC,  Australia. Email: andres.rubianomartinez@monash.edu.} \and Ricardo Ruiz-Baier\thanks{School of Mathematics, Monash University, 9 Rainforest Walk, Clayton 3800 VIC,  Australia; and Universidad Adventista de Chile, Casilla 7-D, Chill\'an, Chile. Email: ricardo.ruizbaier@monash.edu.}}

\begin{document}

\maketitle
\begin{abstract}
    In this work we propose, by linearizing the equations of fully nonlinear poroelasticity, a consistent model in which only the solid phase is incompressible. This reformulation circumvents some inconsistency issues encountered in standard primal formulations of nonlinear poroelasticity while still retaining its key physical coupling mechanisms. We show a well-posed and consistent discretization strategy and also formulate scalable solvers based on a Schur complement formalism. A distinctive feature of the model is that it allows for a lowest order, inf-sup stable family of \ac{fe} spaces. Numerical tests in two and three dimensions are provided to validate the proposed method and solver framework.
\end{abstract}



\begin{acronym}[GPME]
    \acro{amg}[AMG]{Algebraic Multigrid}
    \acro{dofs}[DoFs]{Degrees of Freedom}
    \acro{fe}[FE]{Finite Elements}
    \acro{gpme}[GPME]{Generalized Porous Media Equation}
    \acro{mpgpme}[MP-GPME]{Mixed-in-Pressure Generalized Porous Media Equation}
    \acro{pde}[PDE]{Partial Differential Equation}
    \acrodefplural{pde}[PDEs]{Partial Differential Equation}
\end{acronym}

\section{Introduction}
\paragraph{Context and main contributions.} Poroelastic media refer to a deformable elastic body with a complex network through which fluid passes. This framework has a wide variety of applications, such as CO2 sequestration, hydrogels, and soft tissue \cite{bear1987theory}. The latter is the main motivation of this work, with the additional hypothesis that the solid phase is incompressible \cite{Coussy2004}, as soft tissues are made mostly of water. This hypothesis is seldom used in both linear and nonlinear settings despite its clear interpretation, and instead other options are typically adopted, such as total incompressibility ($J=1$) or quasi-incompressibility \cite{chapelle2010poroelastic,patte2022quasi}. We highlight works on edema modeling \cite{barnafi2022edema}, 
scale and feather growth \cite{barnafi2022coupling}, and heart perfusion \cite{lee2016silico} where solid incompressibility was considered.

The equations of nonlinear poroelasticity, albeit being very general, have remained understudied as they pose several practical difficulties, mainly lack of convergence of the nonlinear solvers involved. One fundamental difficulty was recently detected \cite{barnafi2024fully}, where primal formulations of nonlinear poroelasticity were found to be always inconsistent if the pressure depends on deformation, which is one of the defining physical features of poroelasticity. Two novel mixed formulations were proposed in \cite{barnafi2024fully}, which avoid this problem, and in this work we will focus on the simpler one, consisting in defining a total pressure variable. 

Total pressure has already been considered as a variable in related linearized models \cite{ruiz2022biot}, and has proved to be a reliable reformulation of the problem with much better theoretical properties. It has been used for developing parameter robust methods that do not require an interface Lagrange multiplier \cite{boon2022parameter}, to study nonlinear problems \cite{boon2022parameter,cesmelioglu2023hybridizable}, and to study well-posedness of novel discretization techniques \cite{cesmelioglu2023hybridizable,kumar2024numerical}. Our formulation, despite being different as it does not consider pressure as a variable, depends on a similar idea in which a total pressure is defined as an auxiliary variable. 

In this work, we present a linearization of the fully nonlinear poroelasticity equations with solid incompressibility and a modified total pressure variable. The mathematical structure of the resulting model allows for a simple lowest order \ac{fe} approximation, and we leverage the well-established fixed-stress preconditioner \cite{white2016block,castelletto2016scalable} to obtain a scalable solver for the proposed model.

\paragraph{Notation.} Given $s \geq 0$, the Hilbert space of $\rL^2$ scalar (resp. $\bL^2$ vector) functions with $\rL^2$ derivatives up to order $s$ on a domain $\Omega$ is denoted by $\rH^s(\Omega)$ (resp. $\bH^s(\Omega)$). The corresponding norm and semi-norm of $\rH^s(\Omega)$ are denoted by $\|\bullet\|_{s,\Omega}$ and $|\bullet|_{s,\Omega}$, respectively. When $s=0$, we adopt the convention $\rH^0(\Omega):=\rL^2(\Omega)$ and denote by $(\bullet, \bullet)_\Omega$ the inner product in this space, and use the same notation for its vector-valued counterpart. The subscript is omitted whenever the domain is clear from the context.

On the other hand, The letter $C$ denotes a generic positive constant independent of the mesh size $h$ and physical constants (which might stand for different values at its different occurrences). Furthermore, given any positive expressions $a$ and $b$, the notation $a \,\lesssim\, b$  means that $a \leq C\, b$.

\paragraph{Structure.} This work is structured as follows: Section~\ref{section:model} reviews how a general poroelastic material is characterized in order to have an incompressible solid phase and show how to obtain the sought linearized system. In Section~\ref{section:numerical} we show how the problem can be discretized in a stable manner both in space and time. Here, we propose a Schur complement based preconditioner for the linearized model. In Section~\ref{section:tests} we validate our theoretical claims with convergence tests for the steady and time-dependent versions of the scheme, and also address a benchmark 3D problem. For this, we study separately the optimality and the scalability of the proposed overall solution strategy. We conclude with a brief discussion of the developments of the paper, in Section~\ref{section:discussion}.

\section{Model formulation}\label{section:model}
In this section we present the fully nonlinear model with solid incompressibility in pressure-mixed form. We then compute the linearized model and study the mathematical structure of its weak formulation. We further consider an additional section that shows how the parameters of this linearized model are related to a standard Biot model. This helps position this model within the landscape of other poroelastic models.

\subsection{Nonlinear model with solid incompressibility}\label{section:incompressible_formulation}
For a thorough explanation and derivation of the field equations, we redirect the reader to the excellent book by Coussy \cite{Coussy2004}. Given a domain $\Omega$ in $\mathbb R^{d}$ ($d\in\{2,3\}$) with Lipschitz boundary $\partial \Omega$, nonlinear poroelasticity consists in finding a displacement field $\vec d$ and a Lagrangian porosity $\varphi$ such that the conservation of linear momentum and mass hold, i.e.
\begin{subequations}
    \begin{align}\label{eq:momentum-strong}
        \rho_{s,0}\ddot{\vec d} - \vec \dive  \widetilde{\ten P}(\ten F, \varphi) &= \vec f , \\
        \label{eq:mass-strong}
        \dot \varphi - \dive \ten K \grad \widetilde{p}(\ten F, \varphi) &= g,
    \end{align}
\end{subequations}
where $\ten F = \ten I + \grad \vec d$ is the strain tensor, $\rho_{s,0}$ is the reference solid density, $\vec f$ is a given volume load, $g$ is a mass source or sink, and $\ten K$ is a permeability tensor assumed symmetric, bounded, and uniformly positive definite, i.e., there exist constants $0 < \ten K^{\mathrm{min}} \leq \ten \ten K^{\mathrm{max}} < \infty$ such that
$${\ten{K}^{\mathrm{min}} |\vec{\xi}|^2 \leq \vec \xi^{\tt t} \ten K \vec \xi \leq \ten K^{\mathrm{max}} |\vec \xi|^2 \quad \forall \vec \xi \in \mathbb{R}^d.}$$
The material characterization is given through a Helmholtz potential $\Psi$ such that 
$$ \ten P = \parder{\Psi}{\ten F}, \quad p = \parder{\Psi}{\varphi}. $$
Solid incompressibility states that the Lagrangian solid porosity, given by $\varphi_s = J - \varphi$ with $J=\det \ten F$, is constant in time. In particular, we have that 
\begin{equation}\label{eq:solid-incomp} 
    J - \varphi = J_0 - \varphi_0 := \varphi_{s,0},
\end{equation}
for a given porosity $\varphi_0$ and reference Jacobian $J_0$. Enforcing such a constraint is done through the Helmholtz potential, thus we consider a Langrange multiplier $\lambda$ and modify the Helmholtz potential as 
$$ \widetilde{\Psi} = \Psi + \lambda (J - \varphi  - \varphi_{s,0}). $$
This yields the modified Piola stress tensor
$$ \widetilde{\ten P} = \parder{\Psi}{\ten F} + \lambda \cof\ten F = \ten P + \lambda \cof\ten F, $$
where $\cof \ten F= J \ten F^{-T}$, and the modified pressure
$$ \tilde{p} = \parder{\Psi}{\varphi} - \lambda = p - \lambda. $$
We highlight that this problem, as observed in \cite{barnafi2024fully}, is not consistent due to the presence of third order terms. In addition, the multiplier $\lambda$ is an $\rL^2(\Omega)$ function, but requires the computation of its Laplacian in \eqref{eq:mass-strong}. To circumvent this, we use the mixed in pressure formulation proposed in \cite{barnafi2024fully}, which results in the following problem: find $(\vec d, \mu, \lambda, \varphi)$ such that
\begin{subequations}
    \label{eq:nonlinear-mixed-p}
    \begin{alignat}{2}
        \rho_0\ddot{\vec d} - \dive\left(\ten P(\ten F, \varphi) + \lambda J\ten F^{-T}\right) &= \vec f && \quad \text{in}\,\Omega, \\
        \dot \varphi - \dive \ten K \grad \mu &= g && \quad \text{in}\,\Omega, \\
        J - \varphi &= \varphi_{s,0} && \quad \text{in}\,\Omega, \\
        \mu - p(\ten F, \varphi) + \lambda &= 0 && \quad \text{in}\, \Omega,
    \end{alignat}
\end{subequations}
for some given initial data $\vec d(0)=\vec d^0$, $\dot{\vec d}(0)=\vec v^0$, $\varphi(0)=\varphi^0$ and full Dirichlet boundary conditions given by
$\vec d = \vec 0$, $\mu=0$ on $\partial\Omega$.

We firstly note that the nonlinear terms in \eqref{eq:nonlinear-mixed-p} are $\ten P(\ten F, \varphi)$, $\lambda J \ten F^{-T}$, $J$, and $p(\ten F, \varphi)$.  Next, we consider all variables to be perturbed with respect to an equilibrium point as $\zeta = \zeta_0 + \delta \zeta$ for $\zeta \in\{\vec d, \mu, \lambda, \varphi\}$, from which we will compute a model for the increments $\delta \zeta$.  We obtain the following relations by using a first order approximation of all nonlinear terms: 
\begin{align*}
    \ten P(\ten F, \varphi)&\approx \ten P(\ten F_0, \varphi_0) + \parder {\ten P}{\ten F}:\delta \ten F + \parder{\ten P}{\varphi}\delta \varphi, \\
    \lambda J\ten F^{-T} &\approx \lambda_0J_0 \ten F_0^{-T} + \delta \lambda J_0\ten F_0^{-T} + \lambda_0 \delta(J\ten F^{-T}), \\
    J &\approx J_0 + J_0\ten F_0^{-T} : \delta \ten F, \\
    p(\ten F, \varphi) &\approx p(\ten F_0, \varphi_0) + \parder{p}{\ten F} : \delta \ten F + \parder{p}{\varphi}\delta \varphi, 
\end{align*}
where all omitted arguments in the partial derivatives are the reference equilibrium points. We make the following hypotheses on the linearization terms: 
\begin{itemize}
    \item $\vec d_0 = \vec 0$, i.e., small displacements, and $\lambda_0 = 0$. Note that this yields $J_0 = 1$.
    \item $\parder{\ten P}{\ten F}:\delta \ten F$ yields $\ten \sigma(\grad \vec d) = \mathbb C_\mathrm{Hooke} : \ten \varepsilon(\vec d)$, where $\ten\varepsilon$ is the symmetric part of the gradient and $\mathbb C_\mathrm{Hooke}$ is the well-known linear fourth-order Hooke tensor, assumed symmetric, bounded, and uniformly positive definite. Thus, there exist constants $0 < \mathbb C^{\mathrm{min}} \leq \ten \mathbb C^{\mathrm{max}} < \infty$ such that
    $$\mathbb C^{\mathrm{min}} |\ten \tau|^2 \leq \ten \tau^{\tt t} \mathbb C \ten \tau \leq \mathbb C^{\mathrm{max}} |\ten \tau|^2 \quad \forall \ten \tau \in \mathbb{R}^{d\times d}.$$
    This is not strictly a hypothesis but more of an algebraic convenience because the Piola stress tensor is not symmetric, and deriving a linearization that yields the Hooke tensor requires reparameterizing the Piola in terms of the second Piola tensor as a function of the Cauchy strain $\ten C=\ten F^T \ten F$. We avoid these nuances and simply assume the linearized symmetric form for simplicity.
    \item $\parder{\ten P}{\varphi} = \parder{p}{\ten F} = \beta \ten I$ for some positive constant $\beta = \alpha M$, where $M$ is the Biot modulus and $\alpha$ is the Biot coefficient. 
    \item The pressure function is monotone, so $\parder{p}{\varphi} = M  \geq 0$, with $M$ the Biot modulus.
\end{itemize}
The parameter choices are motivated by the standard Biot model; details are given in Section~\ref{section:parameters}. Using these hypotheses, we obtain the linearized model that we propose,   written in terms of displacement $\vec d$, total pore pressure $\mu$, porosity $\varphi$, and incompressibility Lagrange multiplier $\lambda$:
    \begin{subequations}\label{eq:model-strong}
        \begin{alignat}{2}
            \rho_{s,0}\ddot{\vec d} - \dive\left(\ten \sigma(\grad \vec d) +\beta \varphi \ten I + \lambda \ten I\right) &= \vec f &&\quad \text{in}\,\Omega, \\
            \dot \varphi - \dive \ten K \grad \mu &= g &&\quad \text{in}\,\Omega, \label{eq:model-strong-mass}\\
            \mu - \beta \dive \vec d - M \varphi + \lambda &= p_0 &&\quad \text{in}\, \Omega, \label{eq:model-strong-lambda}\\
            \dive \vec d - \varphi&= -\varphi_{0}&&\quad \text{in}\,\Omega, \label{eq:model-strong-constraint}
        \end{alignat}
    \end{subequations}
    together with the boundary and initial conditions specified at the end of Section~\ref{section:incompressible_formulation}. In contrast to the standard Biot model, the proposed model yields both pressure and porosity as primary variables. In addition, solid incompressibility is considered by rigorously deriving it from the original nonlinear problem, and not through a limit process on a material specific constant. Naturally, both processes yield equivalent models, and we establish this in detail in Section~\ref{section:parameters} with respect to the well-known incompressible Biot.
\subsection{Equivalence with Biot's model}\label{section:parameters}
To obtain the parameters, we look at how they can be compared with a classic Biot system. In such models, the volume fraction $\varphi$ is given by the same linearized law we obtained, given by
    $$ \varphi = \alpha \dive \vec d + \frac 1 M p, $$
where $\alpha$ is the Biot coefficient and $M$ is the Biot modulus. If we consider that the function $\varphi \mapsto p(\ten F, \varphi)$ can be inverted for all $\ten F$, then we can linearize the porosity as 
    $$ \varphi = \varphi(\ten F, p) \approx \varphi_0 + \parder{\varphi}{\ten F} : \ten F + \parder{\varphi}{p} p. $$
Through the linearization, we can then consider $\left(\parder{p}{\varphi}\right)^{-1} = \parder{\varphi}{p} = \frac 1 M$ and $\parder{\varphi}{\ten F} = \alpha\ten I$, from which we can readily compute that 
    $$ \alpha\ten I = \parder{\varphi}{p}\parder{p}{\ten F} = \frac 1 M \parder{p}{\ten F}, $$
so that the linearization hypothesis yields $\parder{p}{\ten F} = \alpha M \ten I$.  To further contextualize the proposed model, we demonstrate its exact equivalence to a standard two-field Biot model in the fully incompressible limit. Starting from the steady state of the continuous formulation defined in \eqref{eq:model-strong}, we can eliminate the auxiliary variables $\lambda$ and $\varphi$ to recover a pure displacement-pressure system. From the solid incompressibility constraint \eqref{eq:model-strong-constraint}, we have:
    \begin{equation}\varphi = \dive \vec d + \varphi_0. \label{eq:phi_elim}\end{equation}
    Substituting \eqref{eq:phi_elim} into the total pressure definition \eqref{eq:model-strong-lambda}, we can express the Lagrange multiplier $\lambda$ exclusively in terms of the displacement $\vec d$ and the pore pressure $\mu$:
    \begin{equation}\lambda = p_0 - \mu + \beta \dive \vec d + M(\dive \vec d + \varphi_0) = p_0 - \mu + (\beta + M)\dive \vec d + M\varphi_0. \label{eq:lambda_elim}\end{equation}
Next, we define the total stress tensor of our formulation as $\ten \sigma_{\text{tot}} = \ten \sigma(\nabla d) + \beta\varphi I + \lambda I$. Inserting \eqref{eq:phi_elim} and \eqref{eq:lambda_elim} into this stress tensor yields:
    \begin{align}
        \ten \sigma_{\text{tot}} &= 2\mu_s \ten \varepsilon(\vec d) + \lambda_s (\dive \vec d) \ten I + \beta(\dive \vec d + \varphi_0) \ten I + \big( p_0 - \mu + (\beta + M)\dive \vec d + M\varphi_0 \big) \ten I \nonumber \\
                            &= 2\mu_s \ten \varepsilon(\vec d) + \big( \lambda_s + 2\beta + M \big) (\dive \vec d) \ten I - \mu \ten I + C_0 \ten I, \label{eq:stress_reduced}
    \end{align}
    where $C_0 = \beta\varphi_0 + p_0 + M\varphi_0$ is a constant initial stress state. Equation \eqref{eq:stress_reduced} is exactly the stress tensor of a standard Biot model with a fluid pressure $\alpha_{\text{Biot}}p = \mu$, with the Biot--Willis coefficient $\alpha_{\text{Biot}}$, and a modified volumetric Lamé parameter defined as
    \begin{equation}\tilde{\lambda} = \lambda_s + 2\beta + M. \label{eq:lame_tilde}\end{equation}
    Furthermore, substituting \eqref{eq:phi_elim} into the mass conservation equation \eqref{eq:model-strong-mass} yields:
    \begin{equation}\dive \dot{\vec d} - \dive(\ten K \nabla \mu) = g,\end{equation}
which matches the standard Biot mass equation in the limit where both the solid grains and the fluid are fully incompressible (i.e., $1/M_{\text{Biot}} = 0$). This reduction confirms that while the proposed four-field system introduces additional variables to rigorously impose solid incompressibility and handle nonlinear stress-free configurations, its linearized form governs the exact same physics as a standard incompressible Biot model subject to an augmented volumetric stiffness.

\subsection{Semi-discrete weak formulation}
We start by considering a time partition $0=t^0 < t^1 < ... < t^N = T$ for some final time $T$ with time step $\tau$ given by $t^{i+1} - t^i \eqqcolon \tau$ for all $i$. We denote the approximation $\zeta(t^n) \approx \zeta^n$ and look for the variables $\zeta^n$ at each instant. In addition, we define the abstract spaces $\bV^d:=\bH^1_0(\Omega)$, $\rV^\mu:=\rH^1_0(\Omega)$, $\rV^\varphi:=\rL^2(\Omega)$, and $\rV^\lambda:=\rL^2(\Omega)$; such that the solution space is given by $\bV^* \coloneqq \bV^d\times \rV^\mu\times \rV^\varphi \times \rV^\lambda$. 

Applying an implicit time discretization yields the following semi-discrete weak formulation: given $\vec f^n\in \bL^2(\Omega)$ and $g^n,p_0,\varphi_0\in \rL^2(\Omega)$; find $(\vec d^n, \mu^n, \varphi^n, \lambda^n)\in \bV^*$ at each time instant $n\in \{1, ..., N\}$, such that
\begin{subequations}\label{eq:weak-form}
    \begin{alignat}{2}
        \frac{\rho_0}{\tau^2}(\vec d^n, \vec d^*) + (\mathbb C_\mathrm{Hooke}:\ten\varepsilon(\vec d^n), \ten\varepsilon(\vec d^*)) + \beta(\varphi^n, \dive \vec d^*) + (\lambda^n, \dive \vec d^*) &= \langle \vec f^n, \vec d^*\rangle &&\quad \forall \vec d^* \in \bV^{\vec d}, \\
        (\varphi^n, \mu^*) + \tau(\ten K\grad \mu^n, \grad \mu^*) &= \tau\langle g^n, \mu^* \rangle &&\quad \forall \mu^* \in \rV^{\mu}, \\
        \beta (\dive \vec d^n, \varphi^*) - (\mu^n, \varphi^*) - (\lambda^n, \varphi^*) + M(\varphi^n, \varphi^*) &= -(p_0, \varphi^*)&&\quad \forall \varphi^* \in \rV^{\varphi}, \\
        (\dive \vec d^n, \lambda^*) - (\varphi^n, \lambda^*) &= -(\varphi_0, \lambda^* ) &&\quad \forall \lambda^* \in \rV^{\lambda}.
    \end{alignat}
\end{subequations}
Here we have used a slight abuse of notation as the right-hand side terms $\vec f^n$ and $g^n$ consider also the terms arising from the discretization of the time derivatives. Note that the equation for $\mu^*$ has been scaled in time.

We now define the linear operators $\cA_1:  \bV^{\vec d} \to (\bV^{\vec d})^\prime$, $\cA_2:  \rV^{\mu} \to (\rV^{\mu})^\prime$, $\cA_3:  \rV^{\varphi} \to (\rV^{\varphi})^\prime$ $\cB_1:\bV^{\vec d} \to (\rV^{\varphi})^\prime$, $\cB_2:\bV^{\vec d} \to (\rV^{\lambda})^\prime$, $\cC:  \rV^{\varphi} \to (\rV^{\mu})^\prime$, and $\cD:  \rV^{\varphi} \to (\rV^{\lambda})^\prime$, together with the linear forms $\cF:\bV^{\vec d} \to \mathbb R$, $\cG:\rV^{\mu} \to \mathbb R$, $\cH:\rV^{\varphi} \to \mathbb R$, and $\cI:\rV^{\lambda} \to \mathbb R$; given by
\begin{gather*}
    \langle \cA_1 \vec d^n, \vec d^* \rangle :=  \frac{\rho_0}{\tau^2}(\vec d^n, \vec d^*) + (\mathbb C_\mathrm{Hooke}:\ten\varepsilon(\vec d^n), \ten\varepsilon(\vec d^*)) , \quad 
    \langle \cA_2 \mu^n, \mu^* \rangle := \tau(\ten K\grad \mu^n, \grad \mu^*),\\ 
    \langle \cA_3 \varphi^n, \varphi^* \rangle := M(\varphi^n, \varphi^*),\quad
    \langle \cB_1 \vec d^n, \varphi^* \rangle := \beta(\varphi^*, \dive \vec d^n),\\
    \langle \cB_2 \vec d^n, \lambda^* \rangle := (\lambda^*, \dive \vec d^n),\quad
    \langle \cC \mu^n, \varphi^* \rangle := (\mu^n, \varphi^*),\quad
    \langle \cD \varphi^n, \lambda^* \rangle := - (\varphi^n, \lambda^*),\\
    \cF\vec d^* := \langle \vec f^n, \vec d^*\rangle,\quad
    \cG\mu^* := \tau\langle g^n, \mu^* \rangle,\quad
    \cH\varphi^* := -(p_0, \varphi^*),\quad
    \cI\lambda^* := -(\varphi_0, \lambda^* ).
\end{gather*}
Thus, the operator form of \eqref{eq:weak-form} reads: given $\vec f^n\in \bL^2(\Omega)$ and $g^n,p_0,\varphi_0\in \rL^2(\Omega)$; find $(\vec d^n, \mu^n, \varphi^n, \lambda^n)\in \bV^*$ at each time instant $n\in \{1, ..., N\}$, such that 
\begin{equation}\label{eq:operator_form}
    \begin{bmatrix}
    \cA_1 & 0 & \cB_1^* & \cB_2^* \\
    0 & \cA_2 & \cC^* & 0 \\
    \cB_1 & -\cC & \cA_3 & \cD \\
    \cB_2 & 0 & \cD & 0 
    \end{bmatrix}
    \begin{bmatrix}
        \vec d^n \\
        \mu^n \\
        \varphi^n \\
        \lambda^n
    \end{bmatrix}
    =
    \begin{bmatrix}
        \cF \\
        \cG \\
        \cH \\
        \cI
    \end{bmatrix} 
    \quad \text{in}\, (\bV^*)^\prime.
\end{equation}

\subsection{Semi-discrete well-posedness analysis}\label{section:semi-discrete-wp}
In this section, we establish the well-posedness of the semi-discrete scheme in operator form \eqref{eq:operator_form}, which is equivalent to the weak formulation \eqref{eq:weak-form}, which we observe to have a double saddle-point structure. We illustrate this observation by defining the linear operators $\cA: \bV \to \bV^\prime$ and $\cB: \bV \to \rQ^\prime$; and the linear form $\widetilde{F}:\bV \to \mathbb R$, as
\begin{align*}
    \langle \cA(\vec d^n,\mu^n,\varphi^n), (\vec d^*,\mu^*,\varphi^*) \rangle &:= (\cA_1 \vec d^n + \cB_1^* \varphi^n,\cA_2 \mu^n + \cC\varphi^n, \cB_1\vec d^n - \cC\varphi^n + \cA_3\varphi^n), \\
    \langle \cB(\vec d^n,\mu^n,\varphi^n), \lambda^*\rangle &:= \cB_2 \vec d^n + \cD\varphi^n,\\
    \widetilde{\cF}(\vec d^*,\mu^*,\varphi^*) &:= (\cF \vec d^*, \cG\mu^*, \cH\varphi^*),
\end{align*}
where $\bV:= \bV^{\vec d}\times \rV^\mu \times \rV^\varphi$ and $\rQ := \rV^\lambda$ equipped with the natural norms 
$$\|(\vec d^*,\mu^*,\varphi^*)\|_{\bV}^2 := \|\vec d^*\|_{1,\Omega}^2 + \|\mu^*\|_{1,\Omega}^2 + \|\varphi^*\|_{0,\Omega}^2 \qquad \text{and} \qquad \|\lambda^*\|_{\rQ}^2 := \|\lambda^*\|_{0,\Omega}^2,$$ respectively. These definitions lead to the following saddle-point formulation: given $\vec f^n\in \bL^2(\Omega)$ and $g^n,p_0,\varphi_0\in \rL^2(\Omega)$; find $((\vec d^n, \mu^n, \varphi^n), \lambda^n)\in \bV\times \rQ$ at each time instant $n\in \{1, ..., N\}$, such that 
\begin{equation}\label{eq:operator_form_simplified}
    \begin{bmatrix}
    \cA & \cB^*\\
    \cB & 0 
    \end{bmatrix}
    \begin{bmatrix}
        (\vec d^n, \mu^n, \varphi^n) \\
        \lambda^n
    \end{bmatrix}
    =
    \begin{bmatrix}
        \widetilde{\cF} \\
        \cI
    \end{bmatrix} 
    \quad \text{in}\, (\bV\times \rQ)^\prime.
\end{equation}

We now focus on checking the Ladyzhenskaya–Babu\v ska–Brezzi conditions \cite{BoffiBrezziFortin2013,gatica2014simple} in order to establish the well-posedness of the saddle-point problem in \eqref{eq:operator_form_simplified}. The following result shows the boundedness of the associated operators.
\begin{lemma}[Boundedness]\label{lem:boundedness}
    The following bounds hold
    \begin{align*}
        \|\cA(\vec d,\mu,\varphi)\|_{\bV^\prime}&\lesssim \max\left\{\frac{\rho_0}{\tau^2},\mathbb C_{\mathrm{Hooke}}^{\mathrm{max}},2\beta,\tau\ten K^{\mathrm{max}},M\right\}\|(\vec d,\mu,\varphi)\|_{\bV},\\
        \|\cB(\vec d,\mu,\varphi)\|_{\rQ^\prime} &\lesssim \|(\vec d,\mu,\varphi)\|_{\bV},\\  \|\cF\|_{\bV^\prime}&\lesssim \|\vec f^n\|_{0,\Omega} + \tau\|g^n\|_{0,\Omega} + \|p_0\|_{0,\Omega},\\
        \|\cG\|_{\rQ^\prime} &\lesssim \|\varphi_0\|_{0,\Omega}.
    \end{align*}
\end{lemma}
\begin{proof}
The bounds readily follow  from the definition of dual norms, together with the Cauchy--Schwarz inequality and the assumptions over $\mathbb C_\mathrm{Hooke}$ and $\ten K$ (cf. Section~\ref{section:incompressible_formulation}). 
\end{proof}
The result below establishes the inf-sup condition, and thus the surjectivity of $\cB$. 
\begin{lemma}[Inf-sup condition]\label{lem:inf-sup}
The following bound holds
    \begin{align*}
        \sup_{(\vec d,\mu,\varphi)\in \bV\setminus\{\vec 0, 0, 0\}}\frac{|\langle\cB(\vec d,\mu,\varphi),\lambda^*\rangle|}{\|(\vec d,\mu,\varphi)\|_{\bV}}   \gtrsim \|\lambda^*\|_{\rQ} \quad \forall \lambda^*\in\rQ.
    \end{align*}    
\end{lemma}
\begin{proof}
From the definition of the operator $\cB$, we readily see that
$$\cB(\vec d,\mu,\varphi) = (\lambda^*, \dive \vec d) - (\varphi, \lambda^*).$$
Since it does not depend on $\mu$, we choose $\vec d = \vec 0$ and $\varphi = -\lambda^*$ to obtain 
\begin{align*}
    \sup_{(\vec d,\mu,\varphi)\in \bV\setminus\{\vec 0, 0, 0\}}\frac{|\langle\cB(\vec d,\mu,\varphi),\lambda^*\rangle|}{\|(\vec d,\mu,\varphi)\|_{\bV}} \geq \frac{|(\lambda^*, \dive \vec 0) - (-\lambda^*, \lambda^*)|}{\|(\vec 0,-\lambda^*)\|_{\bV^{\vec d}\times \rV^{\varphi}}} = \|\lambda^*\|_{\rQ}.
\end{align*}
Thus, the result holds with the hidden constant equal to 1.
\end{proof}
We highlight that the proof hinges on the fact that $V^\lambda$ is contained in $V^\varphi$. This will be the only assumption required for the discrete inf-sup.  Now, we characterize the kernel of the operator $\cB$, denoted as $\ker(\cB)$.
\begin{lemma}[Kernel characterization]\label{lem:ker}
Let $\cB$ defined as in Section~\ref{section:semi-discrete-wp}. Then,
    $$\ker(\cB) = \{(\vec d,\mu,\varphi)\in \bV: \dive \vec d = \varphi \}.$$
\end{lemma}
\begin{proof}
    From the definition of $\ker(\cB)$, we readily see that
    \begin{align*}
        \ker(\cB) &= \{(\vec d,\mu,\varphi)\in \bV: (\lambda^*, \dive \vec d) - (\varphi, \lambda^*)=0\quad \forall \lambda^*\in\rQ \}\\
        &=  \{(\vec d,\mu,\varphi)\in \bV:   (\dive \vec d-\varphi, \lambda^*)=0\quad \forall \lambda^*\in\rQ \}.
    \end{align*}
    Since $\dive \vec d$, $\varphi$, and $\lambda^*$ are functions in  $\rL^2(\Omega)$, the result holds.
\end{proof}
The coercivity of the operator $\cA$ in $\ker(\cB)$ is provided next.
\begin{lemma}[Coercivity of $\cA$]\label{lem:coercivity}
The following bound holds
\begin{align*}
    \langle \cA(\vec d,\mu,\varphi), (\vec d,\mu,\varphi) \rangle \gtrsim \min \left\{\frac{\rho_0}{\tau^2},\mathbb C_\mathrm{Hooke}^{\min},2\beta,\tau \ten K^{\mathrm{min}},M\right\} \|(\vec d,\mu,\varphi)\|_\bV^2
\end{align*}
for all $(\vec d,\mu,\varphi)$ in $\ker(\cB)$.
\end{lemma}
\begin{proof}
    From the definitions of the operator $\cA$ and $\ker(\cB)$, we obtain
    \begin{align*}
        \langle \cA(\vec d,\mu,\varphi), (\vec d,\mu,\varphi) \rangle &= \frac{\rho_0}{\tau^2}(\vec d, \vec d) + (\mathbb C_\mathrm{Hooke}:\ten\varepsilon(\vec d), \ten\varepsilon(\vec d)) + 2\beta(\varphi, \varphi) \\
        &\quad + \tau(\ten K\grad \mu, \grad \mu) + M(\varphi, \varphi).
    \end{align*}
Note that the off-diagonal terms given by $\cC$ and $\cC^*$ canceled out because of the structure of $\ker(\cB)$. The result follows by using the definition of the norm $\|\bullet\|_{\bV}$ in the equation above.
\end{proof}
We finalise this section by establishing the well-posedness of the semi-discrete formulation as a direct consequence of Lemmas~\ref{lem:boundedness}-\ref{lem:coercivity} (see also \cite[Theorem 4.2.3]{BoffiBrezziFortin2013}).

\begin{theorem}\label{th:wp}
    Given $\vec f^n\in \bL^2(\Omega)$ and $g^n,p_0,\varphi_0\in \rL^2(\Omega)$, the saddle-point formulation \eqref{eq:operator_form_simplified} admits a unique solution $((\vec d^n, \mu^n, \varphi^n), \lambda^n)\in \bV\times \rQ$ at each time instant $n\in \{1, ..., N\}$. Moreover, the continuous dependence on data holds, i.e.,
    \ifnum\siam=1
    \begin{multline*}\|((\vec d^n, \mu^n, \varphi^n), \lambda^n)\|_{\bV\times \rQ} \\ \lesssim \max\left\{ \frac{1}{\alpha_0}, \frac{2\sqrt{\alpha_1}}{\sqrt{\alpha_0}},\alpha_1 \right\} \left( \|\vec f^n\|_{0,\Omega} + \tau\|g^n\|_{0,\Omega} + \|p_0\|_{0,\Omega} + \|\varphi_0\|_{0,\Omega}\right),\end{multline*}
\else
    $$ \|((\vec d^n, \mu^n, \varphi^n), \lambda^n)\|_{\bV\times \rQ} \lesssim \max\left\{ \frac{1}{\alpha_0}, \frac{2\sqrt{\alpha_1}}{\sqrt{\alpha_0}},\alpha_1 \right\} \left( \|\vec f^n\|_{0,\Omega} + \tau\|g^n\|_{0,\Omega} + \|p_0\|_{0,\Omega} + \|\varphi_0\|_{0,\Omega}\right),$$
\fi
    with $\alpha_0 = \min \left\{\frac{\rho_0}{\tau^2},\mathbb C_\mathrm{Hooke}^{\min},2\beta,\tau \ten K^{\mathrm{min}},M\right\}$ and $ \alpha_1 = \max\left\{\frac{\rho_0}{\tau^2},\mathbb C_{\mathrm{Hooke}}^{\mathrm{max}},2\beta,\tau\ten K^{\mathrm{max}},M\right\}.$
\end{theorem}
\begin{remark}
    The loss of robustness as $C_{\mathrm{Hooke}}^{\mathrm{max}}$ tends to infinity is an expected limitation of the proposed model, due to the deterioration of the conditioning in the nearly incompressible (or nearly rigid) regime and the consequent reduced control of the stress tensor. This behavior is well known in elasticity-type models with extreme parameter values. Such effects may be mitigated by adopting mixed formulations involving an independent approximation of the stress variable.
\end{remark}
\section{Robust and scalable numerical approximation strategy}\label{section:numerical}
In this section, we show stable choices of conforming \ac{fe} spaces for the proposed model. In addition, we show how it is possible to leverage existing iterative methods and the problem's block structure to generate scalable preconditioners.

\subsection{Discrete spaces and a priori error analysis}\label{section:spaces}
For some abstract conforming spaces $\bV^{\vec d}_h\subset \bV^{\vec d}$, $\rV^\mu_h \subset \rV^\mu$, $\rV^\varphi_h \subset \rV^\varphi$, and $\rV^\lambda_h \subset \rV^\lambda$; such that $\bV_h:= \bV^{\vec d}_h \times \rV^\mu_h \times \rV^\varphi_h \subset \bV$ and $\rQ_h:= \rV^\lambda_h\subset \rQ$ (as detailed later in this section), we observe that discrete counterpart of Lemma~\ref{lem:boundedness} follows immediately. On the other hand, to extend the inf-sup condition provided in Lemma~\ref{lem:inf-sup} in the fully-discrete setting, we observe that an additional assumption is needed, as stated below. For that, we will consider the discrete operator given by 
    \begin{equation}\label{eq:Bh}
        \langle \cB_h(\vec d_h, \mu_h, \varphi_h), \lambda_h^*\rangle \coloneqq (\lambda_h^*, \dive \vec d_h) - (\varphi_h, \lambda^*_h). 
    \end{equation}
\begin{lemma}[Discrete inf-sup condition]\label{lem:discrete_inf-sup}
    Let $\cB_h$ be defined as \eqref{eq:Bh} and $\rV^\lambda_h \subseteq \rV^\varphi_h$. Then, the following holds
\begin{gather*}
    \sup_{(\vec d_h,\mu_h,\varphi_h)\in \bV_h\setminus\{\vec 0, 0, 0\}}\frac{|\langle\cB_h(\vec d_h,\mu_h,\varphi_h),\lambda^*_h\rangle|}{\|(\vec d_h,\mu_h,\varphi_h)\|_{\bV}} \ge \|\lambda^*_h\|_{\rQ} \quad \forall \lambda^*_h\in\rQ_h.
\end{gather*}
\end{lemma}
\begin{proof}
    In virtue of the condition $\rV^\lambda_h \subseteq \rV^\varphi_h$, the proof follows exactly as in Lemma~\ref{lem:inf-sup}, further details are omitted. 
\end{proof}
Notice that the abstract setting we propose does not allow us to characterize the discrete kernel of $\cB$ as in Lemma~\ref{lem:ker}. However, the coercivity of $\cA$ in the discrete kernel still holds.
\begin{lemma}[Discrete coercivity of $\cA$]\label{lem:discrete_coercivity}
Let $\rV^\varphi_h \subseteq \rV^\lambda_h$. Then, the following bound holds
\begin{multline*}
    \langle \cA(\vec d_h,\mu_h,\varphi_h), (\vec d_h,\mu_h,\varphi_h) \rangle \\
        \gtrsim \min\left\{\frac{\rho_0}{\tau^2},C_\mathrm{Hooke}^{\mathrm{min}},2\beta,\tau \ten K^{\mathrm{min}},M\right\} \|(\vec d_h,\mu_h,\varphi_h)\|_\bV^2 \quad \forall (\vec d_h,\mu_h,\varphi_h) \in \ker(\cB_h).
\end{multline*}
\end{lemma}
\begin{proof}
    From the definition of $\ker(\cB_h)$, we obtain that
    \begin{align*}
        \ker(\cB_h) &= \{(\vec d_h,\mu_h,\varphi_h)\in \bV_h: (\lambda^*_h, \dive \vec d_h) - (\varphi_h, \lambda^*_h)=0\quad \forall \lambda^*_h\in\rQ_h \}\\
        &=  \{(\vec d_h,\mu_h,\varphi_h)\in \bV_h:   (\dive \vec d, \lambda^*)=(\varphi_h, \lambda^*_h)\quad \forall \lambda^*_h\in\rQ_h \}.
    \end{align*}
    Then, the assumption $\rV^\varphi_h \subseteq \rV^\lambda_h$ allow us to assert that for all $(\vec d_h,\mu_h,\varphi_h)\in \ker(\cB_h)$ we have $(\dive \vec d, \varphi_h)=(\varphi_h, \varphi_h)$. From here, the proof follows as in Lemma~\ref{lem:coercivity}. 
\end{proof}
Next, we establish the well-posedness of the fully-discrete formulation as a direct consequence of Lemmas~\ref{lem:boundedness}, and \ref{lem:discrete_inf-sup}-\ref{lem:discrete_coercivity}.
\begin{theorem}\label{th:wp_discrete}
    Given $\vec f^n\in \bL^2(\Omega)$ and $g^n,p_0,\varphi_0\in \rL^2(\Omega)$. Assume further that $\rV^\varphi_h = \rV^\lambda_h$. Then, the saddle-point formulation \eqref{eq:operator_form_simplified} restricted to the discrete setting admits a unique solution $((\vec d^n_h, \mu^n_h, \varphi^n_h), \lambda^n_h)\in \bV_h\times \rQ_h$ at each time instant $n\in \{1, ..., N\}$. Moreover, the continuous dependence on data holds, i.e.,
    \ifnum\siam=1
    \begin{multline*}\|((\vec d^n_h, \mu^n_h, \varphi^n_h), \lambda^n_h)\|_{\bV\times \rQ} \\
      \lesssim \max\left\{ \frac{1}{\alpha_0}, \frac{2\sqrt{\alpha_1}}{\sqrt{\alpha_0}},\alpha_1 \right\} \left( \|\vec f^n\|_{0,\Omega} + \tau\|g^n\|_{0,\Omega} + \|p_0\|_{0,\Omega} + \|\varphi_0\|_{0,\Omega}\right),\end{multline*}
  \else
    $$ \|((\vec d^n_h, \mu^n_h, \varphi^n_h), \lambda^n_h)\|_{\bV\times \rQ} \lesssim \max\left\{ \frac{1}{\alpha_0}, \frac{2\sqrt{\alpha_1}}{\sqrt{\alpha_0}},\alpha_1 \right\} \left( \|\vec f^n\|_{0,\Omega} + \tau\|g^n\|_{0,\Omega} + \|p_0\|_{0,\Omega} + \|\varphi_0\|_{0,\Omega}\right),$$
  \fi
    with
    $$\alpha_0 = \min\{\frac{\rho_0}{\tau^2},C_\mathrm{Hooke}^{\min},2\beta,\tau \ten K^{\mathrm{min}},M\} \quad \text{and} \quad \alpha_1 = \max\left\{\frac{\rho_0}{\tau^2},\mathbb C_{\mathrm{Hooke}}^{\mathrm{max}},2\beta,\tau\ten K^{\mathrm{max}},M\right\}.$$
\end{theorem}
Moreover, standard arguments (see, e.g.,  \cite[Theorem 2.5]{gatica2014simple}) provide the following C\'ea estimate for the total error $e:=\|((e(\vec d^n),e(\mu^n),e(\varphi^n)),e(\lambda^n))\|_{\bV\times\rQ}=\|((\vec d^n-\vec d^n_h,\mu^n-\mu^n_h,\varphi^n-\varphi^n_h),\lambda^n-\lambda^n_h)\|_{\bV\times\rQ}$.
\begin{theorem}[C\'ea estimate]\label{th:cea}
    Under the assumptions of Theorems~\ref{th:wp} and \ref{th:wp_discrete}, there holds 
    $$e \lesssim  \inf_{((\vec d^*_h,\mu^*_h,\varphi^*_h),\lambda^*_h)\in\bV_h\times \rQ_h} \|((\vec d^n-\vec d^*_h,\mu^n-\mu^*_h,\varphi^n-\varphi^*_h),\lambda^n-\lambda^*_h)\|_{\bV\times\rQ}.$$
\end{theorem}

\begin{remark}
    Thanks to Lemma~\ref{lem:discrete_coercivity}, we note that there is no need to satisfy an standard inf-sup condition for the $(\dive \vec d^n_h, \varphi^*_h)$ form. Consequently, the lowest-order \ac{fe} pairs $\bP_1^{\text{cont}}$-$\rP_1^{\text{disc}}$ or $\bP_1^{\text{cont}}$-$\rP_0^{\text{disc}}$ may be utilized for the discrete space $\bV^{\vec{d}}_h \times \rV^\varphi_h$. In these instances, the associated test space $\rV^\lambda_h$ must be chosen as continuous piece-wise linear polynomials ($\rP_1^{\text{cont}}$) or discontinuous piece-wise constant polynomials ($\rP_0^{\text{disc}}$), respectively. Alternatively, classical inf-sup stable pairs such as the MINI or Taylor-Hood elements may be employed; for both choices, the corresponding test space $\rV^\lambda_h$ is the space of continuous piecewise linear polynomials. While the definition of $\rV^\mu_h$ remains arbitrary, we practice numerical consistency by choosing a discrete space with the same approximation properties to those used for the other three variables.
\end{remark}

We finalise by specifying the discrete spaces as $\bV^{\vec d}_h=\mathbf{P}_{k+1}^\text{cont}$, $\rV^\mu_h=\mathrm{P}_{k+1}^\text{cont}$, $\rV^\varphi_h=\mathrm{P}_{k}^\text{disc}$, and $\rV^\lambda_h=\mathrm{P}_{k}^\text{disc}$ with $k\geq 0$. Here $\mathbf{P}_{k+1}^\text{cont}$ (resp. $\mathrm{P}_{k}^\text{cont}$) is the set of vector (resp. scalar) valued continuous polynomials of order $k+1$ and $\mathrm{P}_{k}^\text{disc}$ is the set of scalar valued discontinuous polynomials of order $k$. Moreover, we provide the convergence rates of the scheme with this specification.
\begin{theorem}\label{th:cv}
    Under the assumptions of Theorem~\ref{th:cea}. Assume further that $\vec d^n \in \bH^{s+1}(\Omega)$, $\mu^n \in \rH^{s+1}(\Omega)$, $\varphi^n \in \rH^{s}(\Omega)$, and $\lambda^n \in \rH^{s}(\Omega)$ with $0<s\leq k+1$. Then, the following estimate holds:
    $$e\lesssim h^s \left( |\vec d^n|_{s+1,\Omega}  + |\mu^n|_{s+1,\Omega} + |\varphi^n|_{s,\Omega} + |\lambda^n|_{s,\Omega}\right).$$
\end{theorem}
   
\subsection{Robust preconditioner formulation}\label{section:preconditioner}
We propose using a block partitioned preconditioner for this problem, and in particular we use the Schur complement approximation  produced by the fixed stress method \cite{settari1998coupled}. Indeed, the fixed stress iteration is an Uzawa iteration \cite{storvik2020fixed} which can be used as a preconditioner \cite{white2016block}, where we observe that the following approximation is used for the Schur complement: 
    $$ -\dive \left (\sigma(\bullet)\right )^{-1} \grad \approx \beta_{FS} I, $$
    for some problem dependent constant $\beta_{FS}$ (see \cite{both2017robust} for analytical estimates in Biot). We refer to this approximation as the \emph{fixed-stress approximation}, and consider in what follows that $\beta_{FS}=2\mu_s / \text{dim} + \lambda_s$ \cite{both2017robust}, where $\mu_s,\lambda_s$ are the Lamè parameters. Note that the constant $\beta_{FS}$ is necessarily positive, see \cite{both2022iterative} for an example with a negative stabilization constant. Considering the partition of variables into $\rV^d$ and $\rV^\mu\times \rV^\varphi \times \rV^\lambda$, we obtain the following Schur complement block:
    \ifnum\siam=1
    \begin{multline*} \ten S = \underbrace{\begin{bmatrix}\mathcal A_2 & \mathcal C^* & \ten 0 \\ -\mathcal C & \mathcal A_3 & \mathcal D \\ \ten 0 & \mathcal D & \ten 0 \end{bmatrix}}_{\eqqcolon\mathcal H}  - \begin{bmatrix} 0\\ \mathcal B_1 \\ \mathcal B_2 \end{bmatrix} (\mathcal A_1)^{-1} \begin{bmatrix} 0 & \mathcal B_1^* & \mathcal B_2^* \end{bmatrix} \\
    = \mathcal H - \begin{bmatrix} 0\\ \beta\mathcal B \\ \mathcal B \end{bmatrix} (\mathcal A_1)^{-1} \begin{bmatrix} 0 & \beta\mathcal B^* & \mathcal B^* \end{bmatrix}, 
    \end{multline*}
\else
    $$ \ten S = \underbrace{\begin{bmatrix}\mathcal A_2 & \mathcal C^* & \ten 0 \\ -\mathcal C & \mathcal A_3 & \mathcal D \\ \ten 0 & \mathcal D & \ten 0 \end{bmatrix}}_{\eqqcolon\mathcal H}  - \begin{bmatrix} 0\\ \mathcal B_1 \\ \mathcal B_2 \end{bmatrix} (\mathcal A_1)^{-1} \begin{bmatrix} 0 & \mathcal B_1^* & \mathcal B_2^* \end{bmatrix} = \mathcal H - \begin{bmatrix} 0\\ \beta\mathcal B \\ \mathcal B \end{bmatrix} (\mathcal A_1)^{-1} \begin{bmatrix} 0 & \beta\mathcal B^* & \mathcal B^* \end{bmatrix}, $$
\fi
    where we have defined $\mathcal B \coloneqq \mathcal B_2$ and used that $\mathcal B_1 = \beta \mathcal B_2$.  We focus on the second term now, which is where we will obtain the blocks for the fixed-stress approximation $\mathcal B\mathcal A_1^{-1}\mathcal B^* \approx \beta_{FS}\mathcal I$: 
    \ifnum\siam=1
\begin{multline}
\begin{bmatrix} 0\\ \beta \mathcal B \\ \mathcal B\end{bmatrix} (\mathcal A_1)^{-1} \begin{bmatrix} 0 & \beta \mathcal B^* & \mathcal B^* \end{bmatrix} 
    = \begin{bmatrix} 0\\ \beta \mathcal B \\ \mathcal B\end{bmatrix} \begin{bmatrix} 0 & \beta (\mathcal A_1)^{-1}\mathcal B^* & (\mathcal A_1)^{-1} \mathcal B^* \end{bmatrix}  \\
    = \begin{bmatrix} 0 & 0 & 0\\ 0 & \beta^2 \mathcal S_2 & \beta \mathcal S_2 \\ 0 & \beta \mathcal S_2 & \mathcal S_2 \end{bmatrix},
\end{multline}
\else
    $$ \begin{bmatrix} 0\\ \beta \mathcal B \\ \mathcal B\end{bmatrix} (\mathcal A_1)^{-1} \begin{bmatrix} 0 & \beta \mathcal B^* & \mathcal B^* \end{bmatrix} 
    = \begin{bmatrix} 0\\ \beta \mathcal B \\ \mathcal B\end{bmatrix} \begin{bmatrix} 0 & \beta (\mathcal A_1)^{-1}\mathcal B^* & (\mathcal A_1)^{-1} \mathcal B^* \end{bmatrix}  \\
    = \begin{bmatrix} 0 & 0 & 0\\ 0 & \beta^2 \mathcal S_2 & \beta \mathcal S_2 \\ 0 & \beta \mathcal S_2 & \mathcal S_2 \end{bmatrix}, $$
\fi
where we used the notation $\mathcal S_2 \coloneqq \mathcal B (\mathcal A_1)^{-1} \mathcal B^*$. Owing to the fixed-stress approximation $\mathcal S_2 \approx \beta_{FS} I$, we obtain that a spectrally equivalent (in practice) block for the global Schur complement problem is 
    \begin{equation}\label{eq:schur}
        \ten S \approx \mathcal H - \beta_{FS} \begin{bmatrix} 0 & 0 & 0\\ 0 & \beta^2 \ten I & \beta \ten I \\ 0 & \beta \ten I & \ten I \end{bmatrix}, 
    \end{equation} 
for some constant $\beta_{FS}$. This translates into adding the following bilinear form to the preconditioner operator:
    \begin{equation}\label{eq:schur-equivalent}
    s((\varphi_h, \lambda_h), (\varphi_h^*, \lambda_h^*)) \coloneqq \beta_{FS}(\beta \varphi_h + \lambda_h, \beta \varphi_h^*+ \lambda_h^*).
    \end{equation}

\section{Numerical tests}\label{section:tests}
In this section we propose numerical tests that validate  our claims regarding the convergence of the chosen \ac{fe} spaces and the scalability of the solution strategy proposed. For this aim, we use the swelling test as in \cite{both2022iterative}, which consists in a porous medium with a constant pressure on one side. To allow for an easier comparison with other works, we use the parameters from the benchmark provided in \cite{anselmann2023benchmark}. We thus solve \eqref{eq:weak-form} on the unit cube $\Omega=(0,1)^3$ with $\vec f=\vec 0$, $\Theta = 0$, and boundary conditions given as follows: The solid is allowed to slide, so $d_\zeta = 0$ on $\zeta=0$ for $\zeta \in\{x,y,z\}$, with the tangential components being assumed to have homogeneous Neumann conditions. All the other sides are stress free. Regarding the pressure, we set it to $1\,\texttt{kPa}$ on the left $(x=0)$ and front $(y=0)$;  and 0 on the right $(x=1)$ and back $(y=1)$. On the other sides we set homogeneous Neumann conditions, which imply that fluid does not leave the geometry through those surfaces in virtue of the Darcy law: $\vec u_f = -\ten K \grad \mu$. 

Given that we can recast the  parameters used herein through Biot parameters, we use the standard Biot parameters as shown in Table~\ref{table:swelling-params}, where we note that the Young modulus and Poisson ratio give the Lamé parameters according to the relations {$ \hat{\mu} = \frac{E}{2(1+\nu)}$ and $\hat{\lambda} = \frac{E\nu}{(1+\nu)(1 - 2\nu)}$}. Note also that the reference pressure and porosity are not typical in Biot models, and indeed the Biot--Willis coefficient is understood to depend on the porosity. We leave these modeling aspects for future investigations, and simply consider them as additional parameters.

We thus propose three different tests:
    \begin{itemize}
        \item A comparison test where we verify that our model is equivalent to Biot's model using a Mandel test.
        \item A set of convergence tests where we validate the proposed convergence theory.
        \item A set of preconditioning tests where we validate our block preconditioners with fixed-stress approximation. All tests are done with respect to varying \ac{dofs} to verify optimality.
    \end{itemize}
All codes were developed using the Firedrake library \cite{rathgeber2016firedrake}.

\begin{table}[t!]
    \centering
    \begin{tabular}{c||c|c}
    \toprule
    Parameter     & Symbol      & Value  \\ \midrule
    Density       & $\rho_0$    & 1\\
    Young Modulus & $E$         & $2\cdot 10^4$  \\
    Poisson ratio & $\nu$       & 0.3 \\
    Biot--Willis   & $\alpha$    & 0.9    \\
    Permeability  & $\ten K$    & $\ten I$ \\
    Biot modulus  & $M$         & $100$   \\ 
    Reference pressure & $p_0$  & 0 \\
    Reference porosity & $\varphi_0$ & 0.1 \\ \bottomrule
    \end{tabular}
    \caption{Parameters used in swelling test.}
    \label{table:swelling-params}
\end{table}

\subsection{Equivalence with the standard incompressible Biot model}
\label{sec:mandel_equivalence}

To numerically verify the exact mathematical reduction presented in Section~\ref{section:parameters}, we consider a steady-state analogue of the classic Mandel's problem. The physical domain is a two-dimensional block $\Omega = (0, 2) \times (0, 1)$. A uniform downward compressive traction $\mathbf{T} = (0, -1000)$ is applied on the top boundary. We impose roller boundary conditions on the left and bottom edges ($d_x = 0$ at $x=0$, and $d_y = 0$ at $y=0$), while fluid drainage is permitted exclusively through the right boundary ($\mu = 0$ at $x=2$). All other boundaries are assumed impermeable.

To strictly satisfy the inf-sup condition in the fully incompressible limit in both models without resorting to artificial stabilization, we use Taylor--Hood \ac{fe} spaces ($\bP_2^{cont}$ for the displacement vector and $\rP_1^{cont}$ for the scalar fields $\mu, \varphi$, and $\lambda$). The stability of this choice is a consequence of Theorem~\ref{th:cea}. The resulting system was solved using an LU factorization on both models.

Figure \ref{fig:mandel_comparison} illustrates the pressure fields obtained using the proposed four-field formulation ($\mu$) and the standard two-field Biot model ($p$). The spatial distribution and magnitudes are visually indistinguishable.

\begin{figure}[htbp]
    \centering
    \begin{subfigure}{0.48\textwidth}
        \centering
        \includegraphics[width=\linewidth]{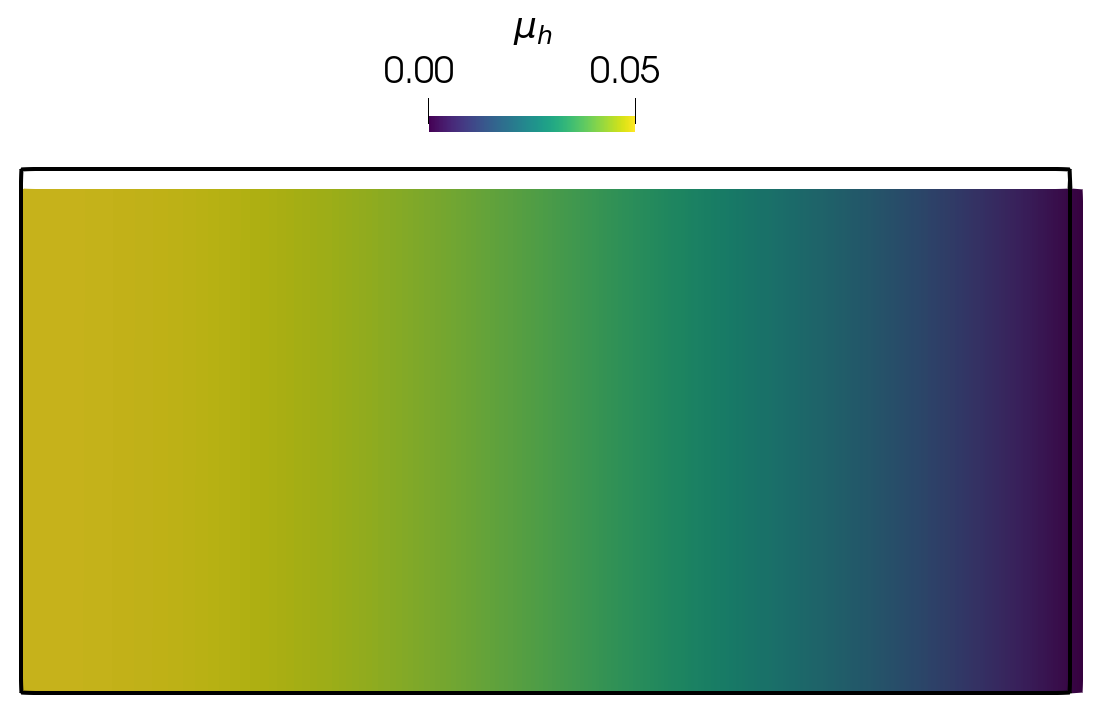}
        \caption{Proposed model ($\mu$)}
        \label{fig:mandel_mu}
    \end{subfigure}
    \hfill
    \begin{subfigure}{0.48\textwidth}
        \centering
        \includegraphics[width=\linewidth]{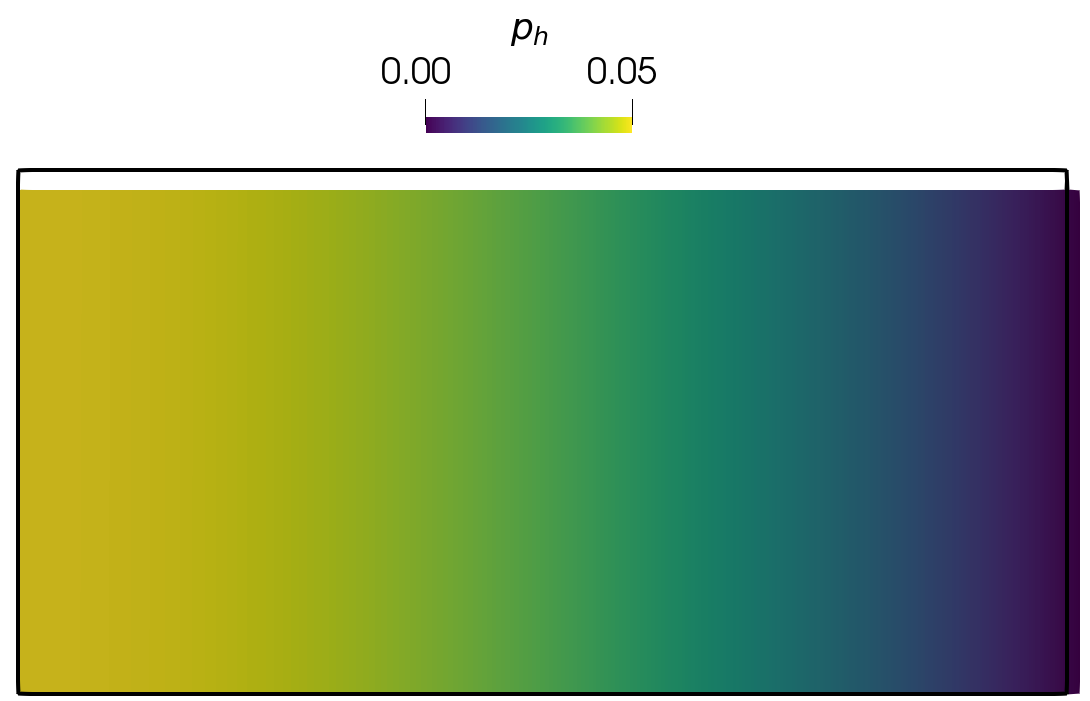}
        \caption{Standard Biot model ($p$)}
        \label{fig:mandel_p}
    \end{subfigure}
    \caption{Comparison of the pressure fields in the steady-state Mandel test. The solutions obtained from both models are visually identical.}
    \label{fig:mandel_comparison}
\end{figure}

To rigorously quantify this equivalence, we compute the $\rL^2$-norm of the difference between the primary variables of the proposed model $(\vec d, \mu)$ and those of the standard Biot model $(\vec u, p)$ over the computational mesh. As reported in Table \ref{tab:mandel_errors}, the discrepancies are on the order of machine precision. This confirms that the proposed linearized formulation and the fully incompressible standard Biot model are equivalent.

\begin{table}[htbp]
    \centering
    \begin{tabular}{l|c}
        \toprule
        \textbf{Variable} & \textbf{$\rL^2$-norm difference} \\
        \midrule
        Displacement ($\|\vec d - \vec u\|_{0,\Omega}$) & $3.14 \times 10^{-15}$ \\
        Pressure ($\|\mu - p\|_{0,\Omega}$)   & $1.06 \times 10^{-14}$ \\
        \bottomrule
    \end{tabular}
    \caption{$\rL^2$ difference between the proposed model and the standard incompressible Biot model for the steady Mandel test.}
    \label{tab:mandel_errors}
\end{table}

\subsection{Convergence rates}
First, we evaluate numerically the convergence in space focusing on the steady case. For this we consider the   domain $\Omega = (0,1)\times(0,1)\times(0,2)$ and manufactured displacement, porosity, and Lagrange multipliers as follows 
\begin{gather*}
\vec d(x,y,z) = \frac14\begin{pmatrix}
    \sin(x)\cos(y)\sin(\frac12z)  \\
    -2\cos(x)\sin(y)\cos(\frac12z)\\
    2\cos(x)\cos(y)\sin(\frac12z)
\end{pmatrix}, \quad 
\mu(x,y,z) = \sin(x)\cos(y)\sin(\frac12z),\\ 
\varphi(x,y,z) = e^{-x}\sin(y)\cos(\frac12z), 
\quad 
\lambda(x,y,z) = \cos(x) e^{-(y+\frac12z)}.
\end{gather*}
These smooth solutions are used to construct non-homogeneous boundary data $\vec d_{\partial\Omega}$ and $\mu_{\partial\Omega}$ as well as manufactured smooth forcing $\vec f$, mass source $g$, reference pressure $p_0$ and reference volume fraction $\varphi_0$. We take model parameters as $\rho_0 = 1$, $E = 1$, $\nu = \frac14$, $\alpha =0.9$, $M=1$, $\ten K = \ten I$, and construct successively refined partitions of $\Omega$ into tetrahedral elements. We compute approximate solutions using $\mathbf{P}_{k+1}^\text{cont}-\mathrm{P}_{k+1}^\text{cont}-\mathrm{P}_k^\text{disc}-\mathrm{P}_k^\text{disc}$ \ac{fe} with $k\in \{0,1\}$. The error history is reported in Table~\ref{table:convergence-space} using the natural norm for each variable, i.e.  the $\bH^1$ norm for displacement error, the $\mathrm{H}^1$ norm for the error in $\mu$, and the $\rL^2$ norm for the remaining errors. The experimental order of convergence is computed as $\texttt{EoC}  =\log(e{(\bullet)}/\tilde{e}{(\bullet)})[\log(h/\tilde{h})]^{-1}$,
where $e,\tilde{e}$ denote errors generated on two
consecutive  meshes of sizes $h$ and~$\tilde{h}$, respectively. 
The results show convergence rates of $O(h^{k+1})$ for all unknowns, and we also show approximate solutions--computed with the second-order method--in Figure~\ref{fig:convergence-space}. 

\begin{table}[t!]
    \centering
   { \begin{tabular}{r|c||c|c|c|c|c|c|c|c}
    \toprule
    \texttt{DoFs}     & $h$      & $e(\vec d)$ & \texttt{EoC}  & $e(\mu)$ & \texttt{EoC}  & $e(\varphi)$ & \texttt{EoC}  & $e(\lambda)$ & \texttt{EoC}  \\ 
    \midrule 
\multicolumn{10}{c}{$k=0$} \\
\midrule 
   72 & 1.7321 & 3.2e-01 & $\star$ & 3.4e-01 & $\star$ & 2.3e-01 & $\star$ & 4.5e-01 & $\star$ \\ 
   372 & 0.8660 & 1.6e-01 & 1.01 & 1.7e-01 & 0.96 & 1.2e-01 & 0.94 & 2.3e-01 & 0.95 \\
  2436 & 0.4330 & 8.0e-02 & 1.00 & 8.7e-02 & 0.99 & 6.2e-02 & 0.98 & 1.2e-01 & 0.99 \\
 17796 & 0.2165 & 4.0e-02 & 1.00 & 4.4e-02 & 1.00 & 3.1e-02 & 0.99 & 5.9e-02 & 1.00 \\
136452 & 0.1083 & 2.0e-02 & 1.00 & 2.2e-02 & 1.00 & 1.5e-02 & 1.00 & 2.9e-02 & 1.00 \\
1069572 & 0.0541 & 9.9e-03 & 1.00 & 1.1e-02 & 1.00 & 7.7e-03 & 1.00 & 1.5e-02 & 1.00 \\
\midrule 
\multicolumn{10}{c}{$k=1$} \\
\midrule 
   276 & 1.7321 & 4.7e-02 & $\star$ & 8.6e-02 & $\star$ & 3.1e-02 & $\star$ & 5.5e-02 & $\star$ \\ 
  1668 & 0.8660 & 1.3e-02 & 1.83 & 2.2e-02 & 2.00 & 9.1e-03 & 1.75 & 1.5e-02 & 1.87 \\
 11652 & 0.4330 & 3.4e-03 & 1.95 & 5.4e-03 & 2.00 & 2.4e-03 & 1.93 & 3.9e-03 & 1.94 \\
 87300 & 0.2165 & 8.5e-04 & 2.01 & 1.3e-03 & 2.00 & 6.1e-04 & 1.97 & 1.0e-03 & 1.96 \\
676356 & 0.1083 & 2.1e-04 & 2.01 & 3.4e-04 & 2.00 & 1.5e-04 & 1.99 & 2.5e-04 & 1.99 \\
   \bottomrule
    \end{tabular}}
    \caption{{Error history against smooth manufactured solutions in 3D, using  the method $\mathbf{P}_{k+1}^\text{cont}-\mathrm{P}_{k+1}^\text{cont}-\mathrm{P}_k^\text{disc}-\mathrm{P}_k^\text{disc}$ with  $k=0,1$. 
    We tabulate in $e(\bullet)$ the error in the $\bV^\bullet$ norm for each variable,  and the corresponding experimental order of convergence (\texttt{EoC}).}}
    \label{table:convergence-space}
\end{table}

\begin{figure}[t!]
    \centering
    \includegraphics[width=0.325\linewidth]{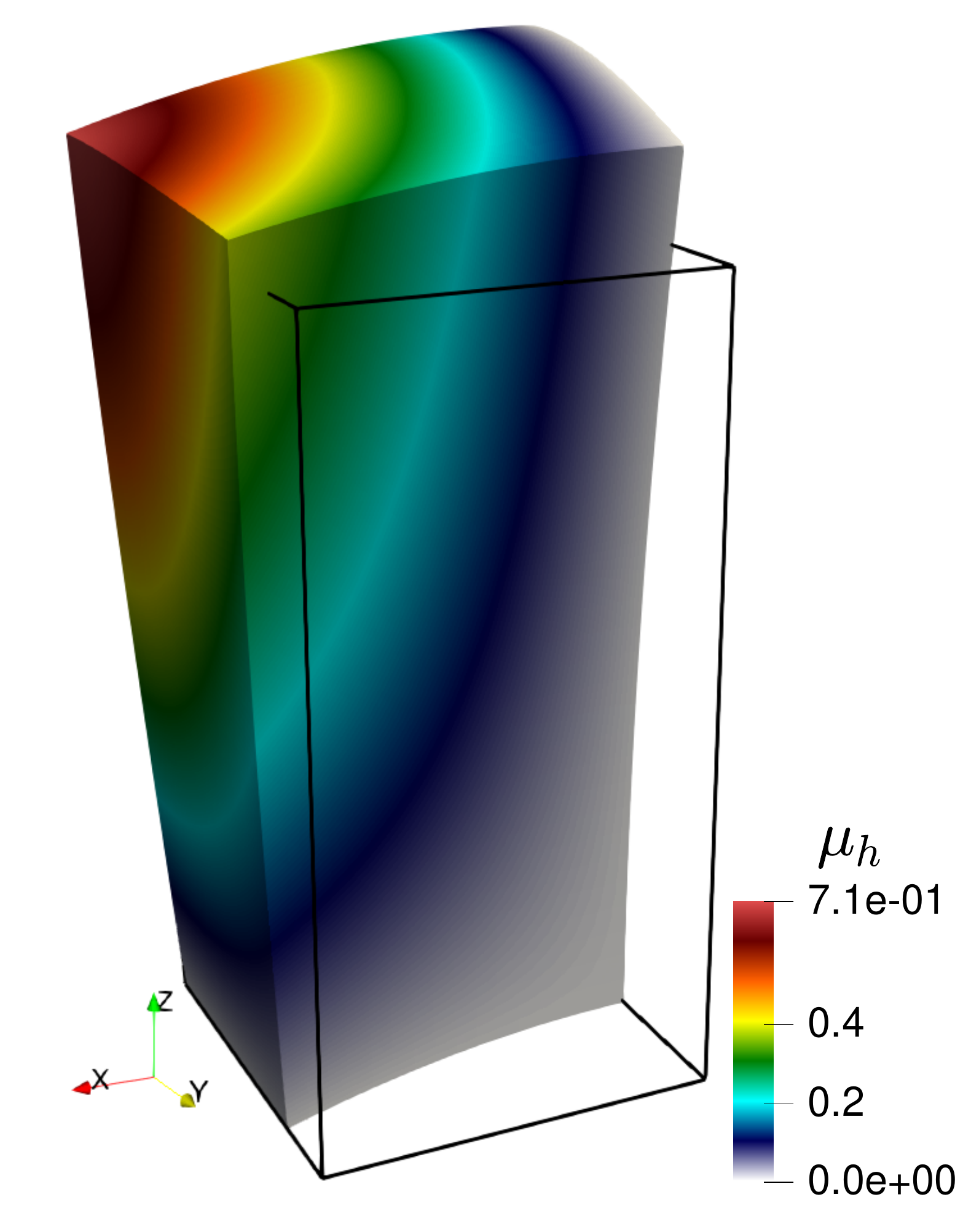}
    \includegraphics[width=0.325\linewidth]{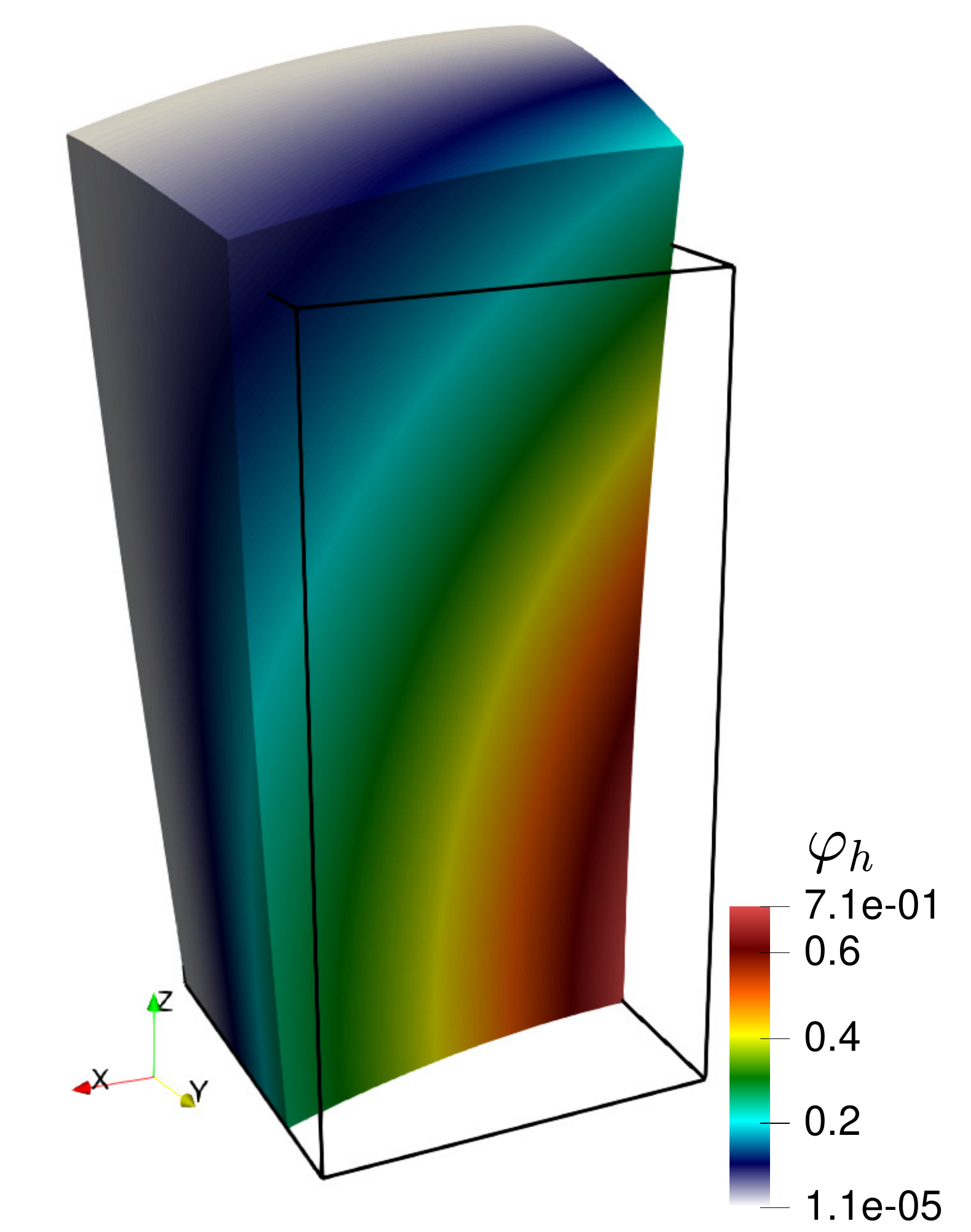}
    \includegraphics[width=0.325\linewidth]{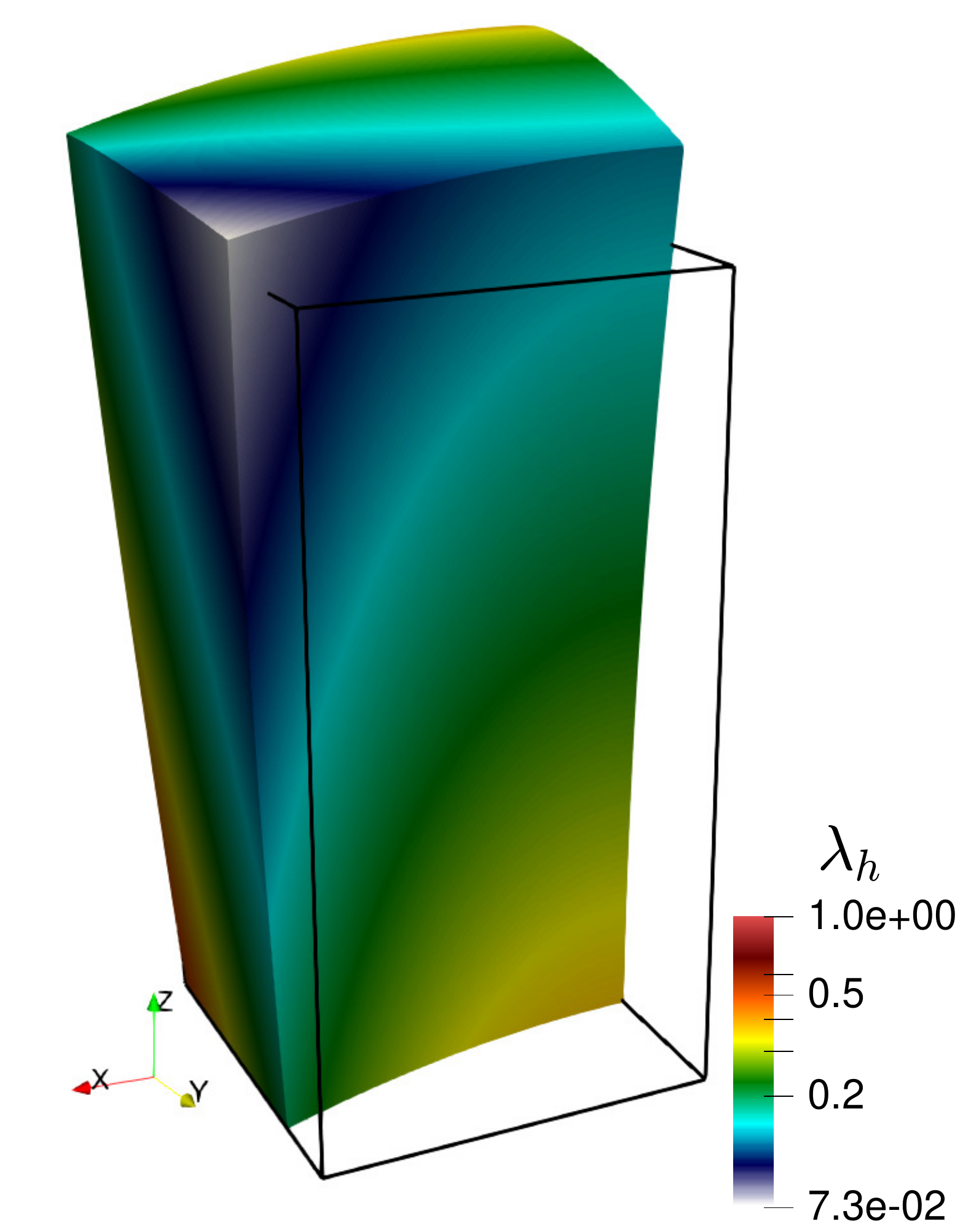}
    \caption{{Approximate solutions for the convergence test in the steady case, computed with the second-order scheme. The Lagrange multipliers and volume fraction are rendered on the deformed configuration,   showing   the outline of the reference domain.}}
    \label{fig:convergence-space}
\end{figure}

We also test the convergence in the $\rL^\infty(0,T; \bV^\bullet)$ norm for all quantities. We do this now in 2D for simplicity, with the domain $\Omega = (0,1)^2$ and time domain $(0,T)$ with (adimensional) final time $T=1$. The model parameters now are $\rho_0 = M = \hat{\lambda} = \hat{\mu} = \beta$, $\ten K = \ten I$ and the manufactured solutions are 
\begin{gather*}
\vec d(x,y) = \frac{t^2}{5}\begin{pmatrix}
    \sin(x)e^{-y} \\
   \cos(x)\sin(y)
\end{pmatrix}, \quad 
\mu(x,y) = t\sin(\pi x)\sin(\pi y),\\ 
\varphi(x,y) = \frac{t}{10}(\cos(\pi [x+y])^2+e^{x+y}), 
\quad 
\lambda(x,y) = \frac{t}{4} \cos(\pi x)\sin(\pi y).
\end{gather*}

We used centered differences to approximate the acceleration term $\ddot{\vec d}$ and backward Euler's scheme for $\dot{\varphi}$, with a constant time step $\tau = 10^{-2}$.  We  report the results obtained with the two lowest-order finite element schemes. Table \ref{table:convergence} confirms an optimal $O(h^{k+1})$ error decay observed for all the field variables, consistent with the theoretical estimates from Theorem \ref{th:cv}. In Figure \ref{fig:convergence} we show approximate solutions at the final time. 

\begin{figure}[t!]
    \centering
    \includegraphics[width=0.3\linewidth]{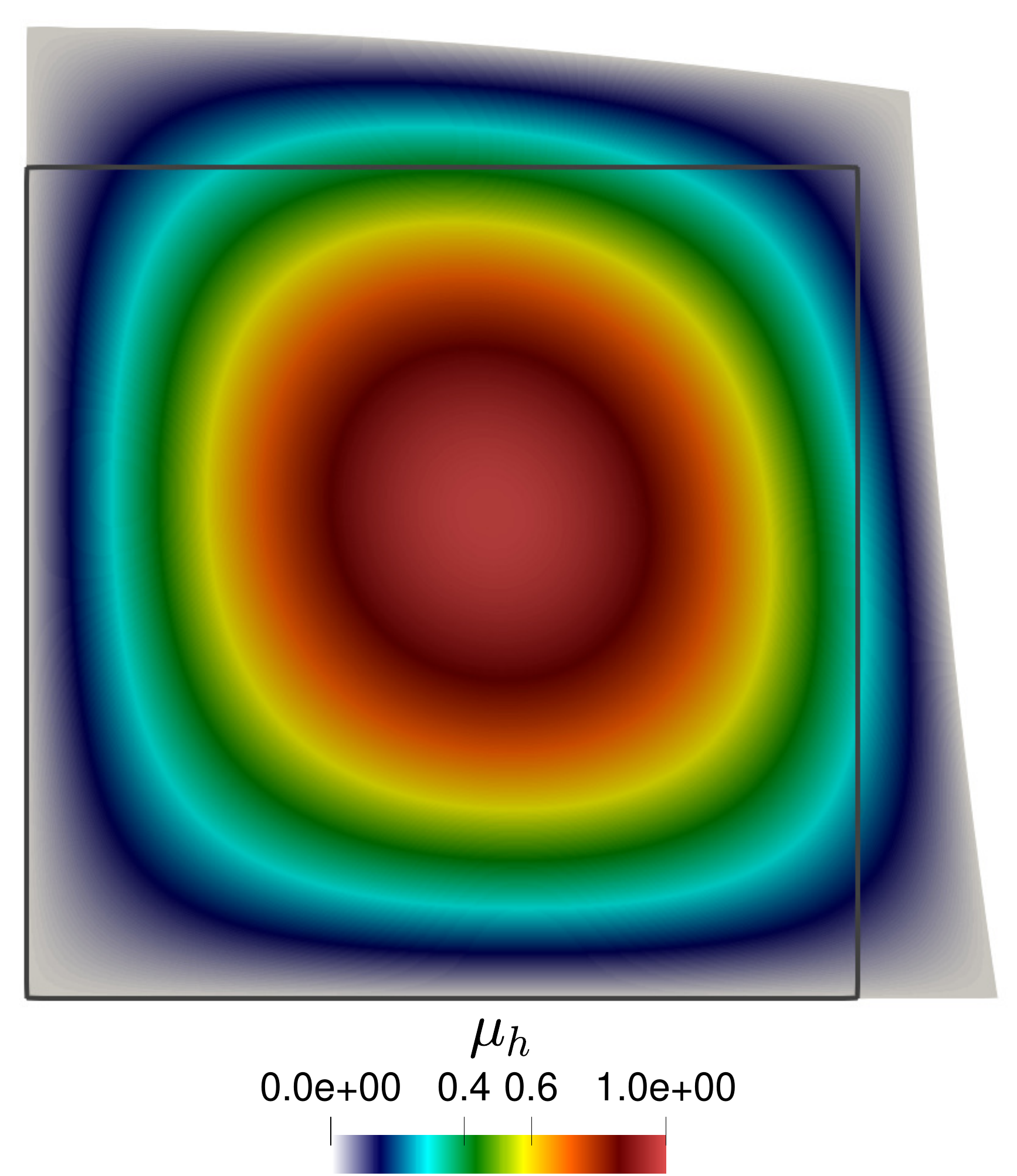}
    \includegraphics[width=0.3\linewidth]{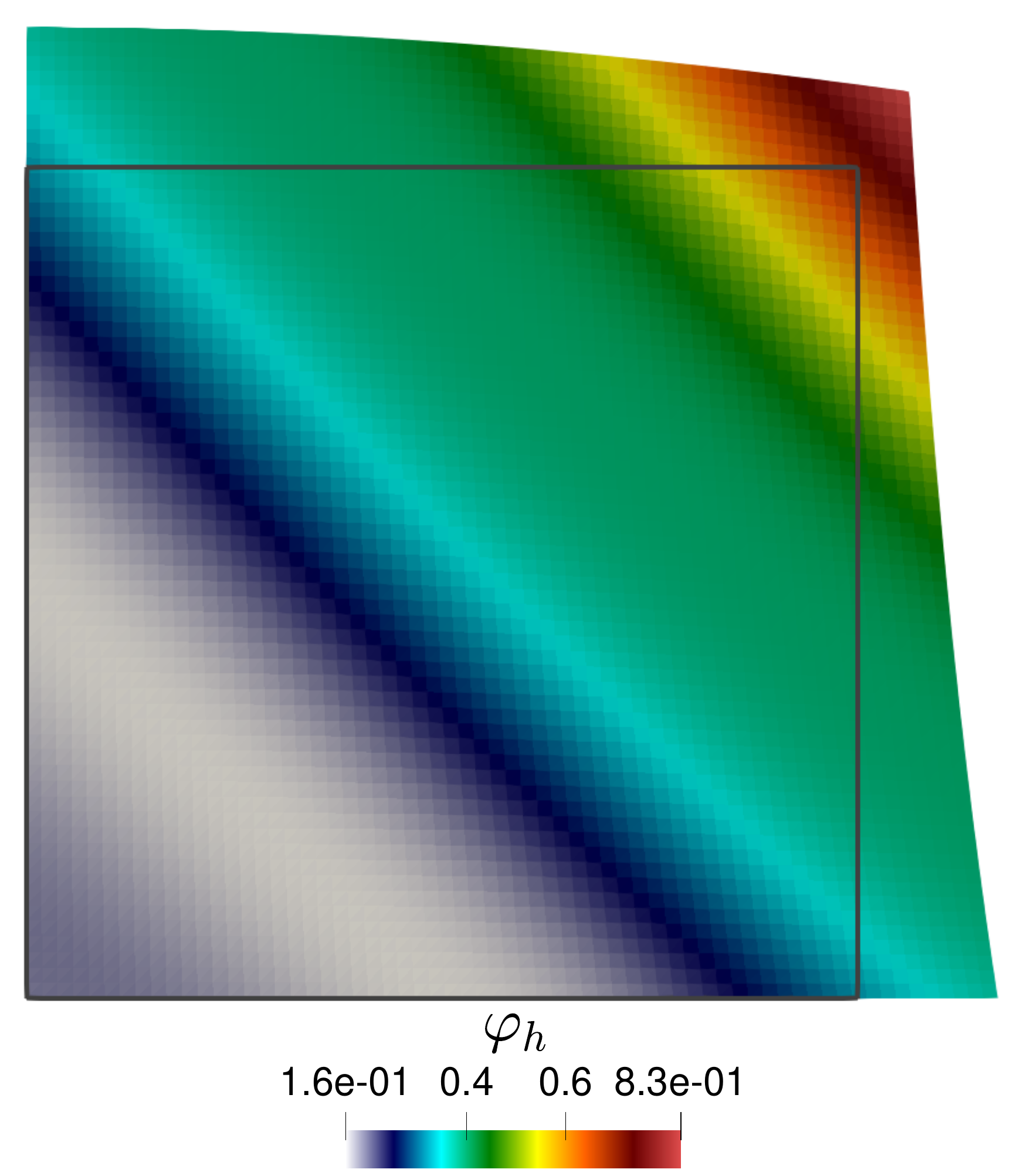}
    \includegraphics[width=0.3\linewidth]{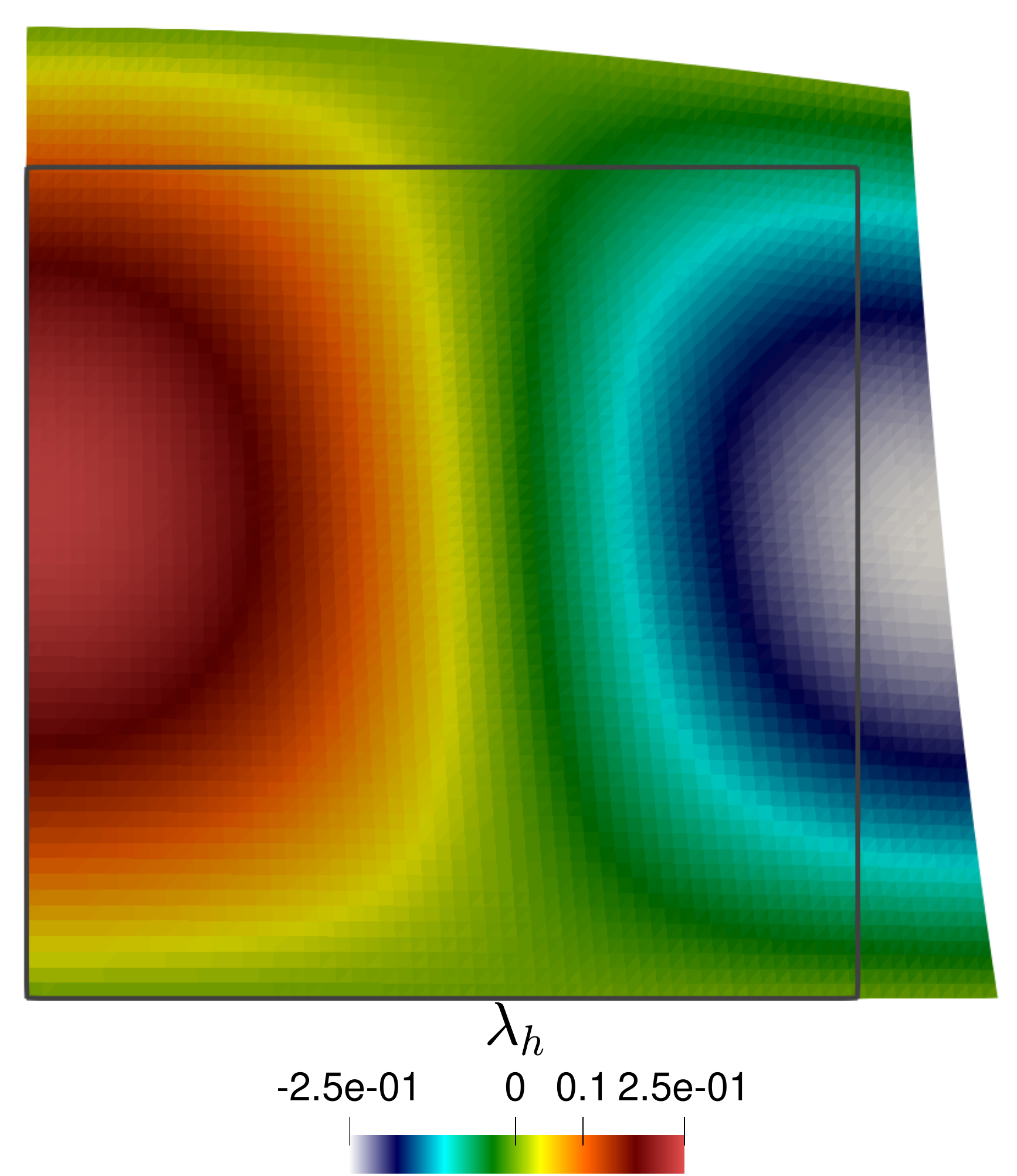}
    \caption{{Approximate solutions for the convergence test in the time-dependent case computed with the lowest order scheme, rendered on the deformed configuration at the final time $T=1$, showing also the outline of the reference domain.}}
    \label{fig:convergence}
\end{figure}

\begin{table}[t!]
    \centering
   { \begin{tabular}{r|c||c|c|c|c|c|c|c|c}
    \toprule
    \texttt{DoFs}     & $h$      & $e(\vec d)$ & \texttt{EoC}  & $e(\mu)$ & \texttt{EoC}  & $e(\varphi)$ & \texttt{EoC}  & $e(\lambda)$ & \texttt{EoC}  \\ 
    \midrule 
\multicolumn{10}{c}{$k=0$} \\
\midrule 
     43 & 0.7071 & 5.1e-02 & $\star$ & 1.5e+00 & $\star$ & 7.2e-02 & $\star$ & 2.3e-01 & $\star$ \\ 
   139 & 0.3536 & 2.7e-02 & 0.93 & 8.4e-01 & 0.85 & 3.9e-02 & 0.86 & 7.5e-02 & 1.60 \\
   499 & 0.1768 & 1.3e-02 & 1.05 & 4.3e-01 & 0.96 & 2.0e-02 & 0.98 & 2.5e-02 & 1.61 \\
  1891 & 0.0884 & 6.5e-03 & 1.02 & 2.2e-01 & 0.99 & 1.0e-02 & 1.00 & 9.7e-03 & 1.35 \\
  7363 & 0.0442 & 3.2e-03 & 1.01 & 1.1e-01 & 1.00 & 5.0e-03 & 1.00 & 4.4e-03 & 1.13 \\
 29059 & 0.0221 & 1.6e-03 & 1.00 & 5.5e-02 & 1.00 & 2.5e-03 & 1.00 & 2.2e-03 & 1.04 \\
\midrule 
\multicolumn{10}{c}{$k=1$} \\
\midrule 
   123 & 0.7071 & 6.0e-03 & $\star$ & 4.7e-01 & $\star$ & 2.9e-02 & $\star$ & 2.8e-02 & $\star$ \\
   435 & 0.3536 & 1.2e-03 & 2.32 & 1.3e-01 & 1.85 & 7.3e-03 & 1.98 & 5.6e-03 & 2.33 \\
  1635 & 0.1768 & 2.7e-04 & 2.16 & 3.3e-02 & 1.96 & 1.9e-03 & 1.92 & 1.3e-03 & 2.10 \\
  6339 & 0.0884 & 6.5e-05 & 2.04 & 8.4e-03 & 1.99 & 4.9e-04 & 1.98 & 3.2e-04 & 2.03 \\
 24963 & 0.0442 & 1.6e-05 & 2.01 & 2.1e-03 & 2.00 & 1.2e-04 & 2.00 & 8.0e-05 & 2.01 \\
 99075 & 0.0221 & 4.0e-06 & 2.00 & 5.3e-04 & 2.00 & 3.1e-05 & 2.00 & 2.0e-05 & 2.00 \\
   \bottomrule
    \end{tabular}}
    \caption{{Error history against smooth manufactured solutions in 2D, using the two lowest order methods $\mathbf{P}_{k+1}^\text{cont}-\mathrm{P}_{k+1}^\text{cont}-\mathrm{P}_k^\text{disc}-\mathrm{P}_k^\text{disc}$ for space discretization,  and centered difference and backward Euler schemes for time discretization with constant time step $\tau = 10^{-2}$.    We tabulate in $e(\bullet)$ the $\rL^\infty(\bV^\bullet)$ error of each variable,  and the corresponding experimental order of convergence (\texttt{EoC}).}}
    \label{table:convergence}
\end{table}

\subsection{Preconditioner validation}
We start by validating that the fixed-stress approximation yields an optimal preconditioner in the setting described at the beginning of Section~\ref{section:tests} in 3D. For this, we consider a lower Schur complement preconditioner with the forms proposed, where all local solves are done with a direct solver. This was implemented using the PETSc \cite{balay2019petsc} options displayed in Appendix~\ref{appendix:petsc-options}. We report the results in Table~\ref{tab:scalability_table}, where we note that the action of the preconditioner is indeed optimal. As the problem is non-symmetric, we used GMRES with a restart of 100, an absolute tolerance of $10^{-12}$ and a relative one of $10^{-8}$. The average number of GMRES iterations remains constant in all tested problem sizes.

\begin{table}[htbp]
    \centering
    \begin{filecontents*}{validation_results.csv}
Solver,DoFs,Avg_Iterations,Avg_Time
semi-direct,204,7.6,0.06940302848815919
semi-direct,1268,9.0,0.07067532539367676
semi-direct,9060,8.0,0.09290103912353516
semi-direct,68804,8.0,0.5300350666046143
semi-direct,536964,8.0,17.53404293060303
\end{filecontents*}
\pgfplotstabletypeset[
        col sep=comma,
        columns={DoFs, Avg_Iterations, Avg_Time}, 
        columns/DoFs/.style={
            column name=\textbf{\ac{dofs}}, 
            int detect,
            column type={r}
        },
        columns/Avg_Iterations/.style={
            column name=\textbf{Avg. Iterations}, 
            fixed, 
            fixed zerofill, 
            precision=1,
            column type={c}
        },
        columns/Avg_Time/.style={
            column name=\textbf{Avg. Time (s)}, 
            fixed, 
            fixed zerofill, 
            precision=3,
            column type={r}
        },
        every head row/.style={before row=\toprule, after row=\midrule},
        every last row/.style={after row=\bottomrule}
    ]{validation_results.csv}
    \vspace{0.2cm}
    \caption{Average linear iterations as a function of the \ac{dofs} during the first 5 time instants.}
    \label{tab:scalability_table}
\end{table}

To obtain an efficient preconditioner, we do not solve blocks exactly, but instead use for each block in the system an adequate preconditioner, so that each matrix inverse is replaced by the action of a preconditioner. The structure of the problem used in Section \ref{section:preconditioner} suggests a block partitioned strategy, which we formulate in the following way in PETSc (see Appendix~\ref{appendix:petsc-options} for the specific instructions): we use a lower Schur complement factorization for the physics-based partitioning given by $V^d$ and $V^\mu\times V^\varphi\times V^\lambda$. For the displacement block, we use the \ac{amg} implementation by HYPRE \cite{falgout2006design} with default parameters. For the second block, we use a block Jacobi preconditioner that separates the fields into the spaces $V^\mu$ and $V^\varphi\times V^\lambda$. For the diagonal operator on $V^\mu$, we use \ac{amg} because of the elliptic operator (the Laplacian), and for the other product block we use a simple Jacobi preconditioner because of the mass matrices. We used GMRES with a restart of 100, an absolute tolerance of $10^{-12}$ and a relative one of $10^{-8}$. The results are shown in Table~\ref{table:efficient}, where we see that the proposed preconditioner is optimal in the considered scenario.

\begin{table}[htbp]
    \centering
    \begin{filecontents*}{scalability_results.csv}
Solver,DoFs,Avg_Iterations,Avg_Time
scalable,1268,55.4,1.8411989212036133
scalable,9060,56.2,0.20388393402099608
scalable,68804,62.0,1.1518883228302002
scalable,536964,66.0,12.563676023483277
scalable,4244228,68.4,123.12406826019287
scalable,33752580,71.4,1096.5285108089447
\end{filecontents*}
\pgfplotstabletypeset[
        col sep=comma,
        columns={DoFs, Avg_Iterations, Avg_Time},
        columns/DoFs/.style={
            column name=\textbf{\ac{dofs}},
            int detect,
            column type={r}
        },
        columns/Avg_Iterations/.style={
            column name=\textbf{Avg. Iterations},
            fixed,
            fixed zerofill,
            precision=1,
            column type={c}
        },
        columns/Avg_Time/.style={
            column name=\textbf{Avg. Time (s)},
            fixed,
            fixed zerofill,
            precision=2,
            column type={r}
        },
        every head row/.style={before row=\toprule, after row=\midrule},
        every last row/.style={after row=\bottomrule}
    ]{scalability_results.csv}
    \vspace{0.2cm}
    \caption{Scalability results showing the degrees of freedom, average GMRES iterations, and average solution time per time step for the fully scalable preconditioner configuration.}
    \label{table:efficient}
\end{table}

\section{Discussion}\label{section:discussion}
In this work we formulate a linear poroelasticity model which accounts for solid phase incompressibility. One important feature of this model is that it is directly derived from a fully nonlinear poroelastic model, which yields a consistent formulation with both the porosity and the total pressure as variables, in addition to the expected displacement and Lagrange multiplier. This model presents a simple saddle-point structure which allows for a lowest order \ac{fe} approximation, with only a very mild inf-sup condition which requires $\rL^2$ functions to belong to related spaces. Finally, the presence of the variable $\mu$ as a total pressure suggests a connection to better established \emph{total-pressure} formulations in the context of Biot's equations, but a rigorous proof does not yet exist. 

In terms of finding a scalable solver for this model, we have leveraged the interpretation of the fixed-stress iterative method as an Uzawa method to obtain a robust preconditioner. The preconditioner obtained relies on a spectrally equivalent block for the Schur complement which is much easier to precondition. For this operator, we have derived efficient options to use for a block-partitioned preconditioner depending ultimately on \ac{amg} and Jacobi. The overall solution strategy was tested on a 3D benchmark problem, where its optimality and scalability were verified numerically.

For future work, the fixed-stress preconditioner has some well-established weaknesses depending on certain parameter regimes, for which it would be interesting to obtain a similar formulation based on the \emph{undrained} iterative scheme. An alternative route would be developing an operator preconditioner based on the continuous spaces. In addition, the solution strategy was tested only on the linearized problem, so it will be important to validate this strategy also on the fully nonlinear case, as well on other systems coupled with poromechanics (such as thermo-poroelasticity \cite{ge2026error}, stress-assisted diffusion, or interface poroelastic-fluidic problems \cite{bansal2026lagrange}).

\bibliography{main} 
\bibliographystyle{alpha}

\newpage
\appendix 

\section{PETSc parameters}\label{appendix:petsc-options}
In this section we specify the main parameter values that we pass to PETSc. These parameters control key aspects of the linear and nonlinear solvers, including the choice of preconditioners, convergence tolerances, and Krylov subspace methods. 

\medskip

	\begin{lstlisting}[language=python, frame=single, caption=PETSc commands for the validation test.]
      "mat_type": "nest",
      "ksp_atol": 1e-12,
      "ksp_rtol": 1e-8,
      "ksp_type": "gmres",
      "ksp_gmres_modifiedgramschmidt": None,
      "ksp_gmres_restart": 100,
      "pc_type": "fieldsplit",
      "pc_fieldsplit_type": "schur",
      "pc_fieldsplit_0_fields": "0",
      "pc_fieldsplit_1_fields": "1,2,3",
      "pc_fieldsplit_schur_fact_type": "full",
      "pc_fieldsplit_schur_precondition": "a11",
      "fieldsplit_0_ksp_type": "preonly",
      "fieldsplit_0_pc_type": "lu",
      "fieldsplit_1_ksp_type": "preonly",
      "fieldsplit_1_pc_type": "lu"
	\end{lstlisting}

	\begin{lstlisting}[language=python, frame=single, caption=PETSc commands for the scalable solver test.]
      "mat_type": "nest",
      "ksp_atol": 1e-12,
      "ksp_rtol": 1e-8,
      "ksp_type": "gmres",
      "ksp_gmres_modifiedgramschmidt": None,
      "ksp_gmres_restart": 100,
      "pc_type": "fieldsplit",
      "pc_fieldsplit_type": "schur",
      "pc_fieldsplit_0_fields": "0",
      "pc_fieldsplit_1_fields": "1,2,3",
      "pc_fieldsplit_schur_fact_type": "lower",
      "pc_fieldsplit_schur_precondition": "a11",
      "fieldsplit_0_ksp_type": "preonly",
      "fieldsplit_0_pc_type": "hypre",
      "fieldsplit_1": {
          "ksp_type": "preonly",
          "pc_type": "fieldsplit",
          "pc_fieldsplit_type": "additive",
          "pc_fieldsplit_0_fields": "0", 
          "pc_fieldsplit_1_fields": "1,2", 
          "fieldsplit_0_ksp_type": "preonly",
          "fieldsplit_0_pc_type": "hypre",
          "fieldsplit_1_ksp_type": "preonly",
          "fieldsplit_1_pc_type": "jacobi"}
	\end{lstlisting}

\end{document}